\documentclass{article}
\usepackage{amsmath, amssymb, amsthm}
\usepackage{graphicx}
\usepackage{geometry}
\usepackage{indentfirst}
\usepackage{enumitem}
\usepackage{biblatex}
\usepackage{color,colortbl, xcolor}
\usepackage{authblk}
\usepackage{hyperref}
\usepackage{tikz}
\usetikzlibrary{calc,positioning,arrows.meta,shapes.geometric,math}
\tikzset{
    buffer/.style={
        draw,
        shape border rotate=-90,
        isosceles triangle,
        isosceles triangle apex angle=60,
        node distance=2cm,
        minimum height=4em
    }
}
\tikzset{>={Latex[width=2mm,length=2mm]}}
\usepackage{ifthen}

\usepackage[capitalize]{cleveref}
\usetikzlibrary{decorations.pathmorphing}

\usepackage{graphicx,caption}
\numberwithin{equation}{section}
\theoremstyle{definition}
\newtheorem{theorem}[equation]{Theorem}
\newtheorem*{theorem*}{Theorem}
\newtheorem{lemma}[equation]{Lemma}
\newtheorem{definition}[equation]{Definition}
\newtheorem*{definition*}{Definition}
\newtheorem{corollary}[equation]{Corollary}
\newtheorem*{corollary*}{Corollary}
\newtheorem{proposition}[equation]{Proposition}
\newtheorem*{proposition*}{Proposition}
\newtheorem{example}[equation]{Example}
\newtheorem{question}[equation]{Question}

\newtheorem{claim}[equation]{Claim}

\AddToHook{env/proposition/begin}{\crefalias{equation}{proposition}}
\AddToHook{env/theorem/begin}{\crefalias{equation}{theorem}}
\AddToHook{env/corollary/begin}{\crefalias{equation}{corollary}}
\AddToHook{env/lemma/begin}{\crefalias{equation}{lemma}}
\AddToHook{env/example/begin}{\crefalias{equation}{example}}
\AddToHook{env/definition/begin}{\crefalias{equation}{definition}}

\crefname{lemma}{lemma}{lemmas}
\crefformat{lemma}{#2lemma~#1#3} 
\Crefformat{lemma}{#2Lemma~#1#3} 

\crefname{theorem}{theorem}{theorems}
\crefformat{theorem}{#2theorem~#1#3} 
\Crefformat{theorem}{#2Theorem~#1#3}

\crefname{proposition}{proposition}{propositions}
\crefformat{proposition}{#2proposition~#1#3} 
\Crefformat{proposition}{#2Proposition~#1#3} 

\crefname{corollary}{corollary}{corollaries}
\crefformat{corollary}{#2corollary~#1#3} 
\Crefformat{corollary}{#2Corollary~#1#3} 

\crefname{example}{example}{examples}
\crefformat{example}{#2example~#1#3} 
\Crefformat{example}{#2Example~#1#3} 

\crefname{definition}{definition}{definitions}
\crefformat{definition}{#2definition~#1#3} 
\Crefformat{definition}{#2Definition~#1#3}

\newcommand{\myref}[1]{\hyperref[#1]{\Cref{#1}}}
\newenvironment{myproof}[2] {\paragraph{\it Proof of {#1} {\ref{#2}}:}}{\hfill$\square$}

\DeclareMathOperator{\restrict}{\upharpoonright}

\DeclareMathOperator{\rk}{rk}

\newcommand{\set}[1]{\left\{ {#1} \right\}}
\newcommand{\seq}[1]{{\left( {#1} \right)}}
\newcommand{\comment}[1]{}
\newcommand{\downward}[1]{\left\lceil {#1} \right\rceil}
\newcommand{\odownward}[1]{\mathring{\left\lceil {#1} \right\rceil}}

\newcommand{\mF}{{\mathcal{F}}}

\newcommand{\mP}{{\mathcal{P}}}

\newcommand{\mR}{{\mathcal{R}}}

\newcommand{\mU}{{\mathcal{U}}}

\newcommand{\rmV}{{\mathrm{V}}}
\newcommand{\rmE}{{\mathrm{E}}}
\newcommand{\rmC}{{\mathrm{C}}}
\newcommand{\rmt}{{\mathrm{t}}}

\title{On covering properties of end and ray spaces}
\author{Rodrigo Rey Carvalho
\thanks{Institute of Mathematics, Statistics and Computer Science.
E-mail: rrc@ime.usp.br}
}
\author{Matheus Duzi
\thanks{Institute of Mathematics and Computer Sciences. E-mail: matheus.duzi.costa@usp.br}
}
\author{Vinicius de Oliveira Rodrigues
\thanks{Institute of Mathematics, Statistics and Computer Science. E-mail: vinior@ime.usp.br}
}
\affil{University of São Paulo}
\begin{document}
\maketitle
\begin{abstract}
We provide new results on combinatorial characterizations of covering properties in end spaces and ray spaces.
In particular, we characterize the Lindelöf degree, the extent, the Rothberger property, $\sigma$-compactness and the Menger property for ray, end and edge-end spaces. We show that $\sigma$-compactness and the Menger property are equivalent for these spaces, and that they are all $D$-spaces.
As an application of some of these characterizations, we are able to provide combinatorial characterizations of graphs with countably many ends and edge-ends.

\end{abstract}
\section{Introduction}

The end space of an infinite graph captures the intuitive concept of the boundary of the graph.
Intuitively, an end is a ``point at the infinity'' that can be reached by following a ray (one-way infinite path) in the graph.
A natural topology can be defined on the set of ends of a graph to better study their properties and give precise meaning to this intuition \cite{halin1964unendliche}.
This topological space is known as the \emph{end space} of the graph.
This topology generalizes the Freudenthal boundary developed by Freudenthal and Hopf in the 1940s for infinite groups \cite{freudenthal1942neuaufbau, hopf1943enden}.

End spaces are easily seen to be Fréchet, zero-dimensional, Hausdorff spaces. Their topological properties and their relationship with the combinatorial structure of the underlying graph have been extensively studied in recent years.
For instance, they have been shown to be normal \cite{sprussel2008endSpaces}, and, later, hereditarily ultraparacompact \cite{kurkofka2021approximating}.
The compactness of end spaces was combinatorially characterized in \cite{diestel2006end}.
Recently, Kurkofka and Pitz \cite{kurkofka2024representationtheoremendspaces} proved that every end space of a graph can be represented as a particular ray space of an order-theoretic rooted tree.
This representation theorem allows us to study end spaces using the combinatorial structure of trees, which are sometimes easier to analyze.
A topological characterization of end spaces was given in \cite{pitz2023characterisingpathraybranch}, where the Lindelöf degree of ray spaces have also been characterized in terms of the combinatorial structure of the underlying tree.
However, it is not easy to translate these results to end spaces in terms of graphs by using the representation theorem of \cite{kurkofka2024representationtheoremendspaces}.

Formally, an end is an equivalence class of rays in a graph.
As properly recalled in the next subsection, two rays are equivalent in the sense of ends if, and only if, no matter how many finitely many vertices are removed from the graph, there is still a path connecting them.
A natural modification of this definition is to consider the removal of finitely many edges instead of vertices.
With this modification, we obtain the notion of edge-ends, which have been recently studied in papers such as \cite{aurichi2024topologicalremarksendedgeend} and \cite{pitz2025metrization}.

In this paper, we provide new results on combinatorial characterizations of covering properties in end spaces, edge-end-spaces and ray spaces.
In particular, we combinatorially characterize the Lindelöf degree and the extent of end-spaces and edge-end spaces of graphs in terms of the underlying graph, and the Rothberger property for end spaces, edge-end spaces of graphs and ray spaces of trees.
We also provide characterizations of the Menger property and $\sigma$-compactness for these spaces, and show that these two properties are equivalent for these classes of spaces.
Moreover, we show that all these classes of spaces are $D$-spaces.

In the next subsections we review some basic definitions and relevant results.
For a list of the main results of the paper, please refer to Section~\ref{subsec:mainresults}.
\subsection{Basic definitions}
In this subsection we fix the notation and terminology used throughout this paper, and review some important basic definitions and relevant results. 

Most concepts that were not defined in this section will be defined when needed.
In particular, we postpone the definition of edge-ends and the edge-end space of a graph to Section 6. 
We refer the reader to \cite{engelking1977general} for general topology and to \cite{diestel2025graph} for graph theory. 

The cardinality of a set $A$ is denoted by $|A|$.
The finite cardinals (the natural numbers) are identified with the von Neumann ordinals, so $0=\emptyset$ and $n+1=\{0, \dots, n\}$ for every $n \in \omega=\mathbb N$.
If $\kappa$ is a (finite or infinite) cardinal number and $A$ is a set, $[A]^\kappa$ is the set of all subsets of $A$ of cardinality $\kappa$.
Likewise, $[A]^{<\kappa}$ and $[A]^{\leq\kappa}$ are the sets of all subsets of $A$ with cardinality less than $\kappa$ and at most $\kappa$, respectively.

$A^B$ is the set of all functions from $B$ into $A$.
If $\alpha$ is an ordinal, $A^{<\alpha}=\bigcup_{\beta<\alpha} A^\beta$ and $A^{\leq \alpha}=\bigcup_{\beta\leq\alpha} A^\beta$.
In particular, $A^{<\omega}$ is the set of all finite sequences of elements of $A$.

\subsubsection{On graphs and ends}
We now fix the notation we will use and state some basic, standard definitions about graphs that are relevant to our work.

\begin{definition}[Graph]
    A \emph{graph} is a pair $G=(V, E)$ where $V$ is a set and $E\subseteq [V]^2$.
    We write $V=V(G)$ and $E=E(G)$.
\end{definition}

\begin{definition}[Subgraphs and induced subgraph]
    Let $G$ be a graph.
    \begin{enumerate}[label=(\roman*)]
    \item A \emph{subgraph} of $G$ is a graph $H=(V', E')$ such that $V'\subseteq V(G)$ and $E'\subseteq E(G)\cap [V']^2$.
    \item If $H\subseteq V(G)$, the \emph{induced subgraph} of $G$ on $H$ is the graph $(H, E(G)\cap [H]^2)$.
    \end{enumerate}
\end{definition}

\begin{definition}[Paths, rays and tails]
    Let $G$ be a graph.
    \begin{enumerate}[label=(\roman*)]
    \item A \emph{path} in $G$ is an injective sequence of vertices $P=(v_0, \dots, v_n)$ such that $n\in \omega$ and for every $i<n$, $\{v_i, v_{i+1}\} \in E$.
    In this case, we say that $P$ is a \emph{path from $v_0$ to $v_n$}, and denote this by $v_0Pv_n$.
    \item A \emph{ray} in $G$ is an injective sequence of vertices $R=(v_0, v_1, \dots)$ such that for every $i\geq 0$, $\{v_i, v_{i+1}\} \in E$.
    \item A \emph{tail} of a ray $R$ is a sequence of the form $(v_n, v_{n+1}, \dots)$ for some $n\geq 0$.
    \end{enumerate}
\end{definition}

\begin{definition}[Connected graphs, components and trees]
    Let $G$ be a graph.
    \begin{enumerate}[label=(\roman*)]
    \item We say that $G$ is \emph{connected} if for every pair of distinct vertices $u, v \in V(G)$, there exists a path from $u$ to $v$.
    \item We say that $G$ is a tree if every two distinct vertices $u, v \in V(G)$ are connected by a unique path in $G$.
    \item A \emph{connected component} of $G$ is a maximal connected subgraph of $G$. Notice that connected components are always induced by its vertices.
    \item A \emph{rooted tree} is a pair $(T, r)$ where $T$ is a tree and $r \in V(T)$ is called the \emph{root} of $T$.
    \item A \emph{subtree} of a rooted tree $(T, r)$ is a rooted tree $(T', r)$ such that $T'$ is an induced subgraph of $T$.
    \end{enumerate}
\end{definition}

\begin{definition}[Tree order]
    Let $(T, r)$ be a rooted tree.
    The \emph{tree order of $T$ (with respect to the root $r$)} is the relation $\leq$ defined as follows:

    For $x, y\in T$, we say that $x\leq y$ if $x$ lies on the unique path from $r$ to $y$.
    In case $x\leq y$ and $x\neq y$, we may write $x<y$.

    For $x \in T$, we define $\odownward{x}=\{y \in T: y< x\}$ and $\downward{x}=\{y \in T: y\leq x\}$.
\end{definition}

It is clear that the tree order is a well-founded partial order on $T$.
Moreover, for every $x \in T$, $\odownward{x}$ is finite, and, if $x' \in T'$ and $(T', R)$ is a subtree of $(T, r)$, then $\{y \in T: y< x\}=\{y \in T': y< x\}$.
\begin{definition}[Normal tree]
    Let $(T, r)$ be a rooted tree, where $T$ is a subgraph of a graph $G$.

    A \emph{$T$-path (with respect to $G$)} is a path $P=(x_0, \dots, x_n)$ in $G$ such that $x_0, x_n \in T$ and for every $i<n$, $x_i\notin T$.
    In this case, we say that $P$ is a \emph{$T$-path from $x_0$ to $x_n$}.

    We say that $T$ is \emph{normal} if whenever $x, y \in T$ are distinct points incompatible in the tree order of $T$, there is no $T$-path from $x$ to $y$ in $G$.
\end{definition}

A very useful property of normal trees is stated in the following well-known lemma.
For a proof, see \cite[Proposition 1.2.1]{diestel2025graph}.

\begin{lemma}
    Let $(T, r)$ be a normal tree of a graph $G$.
    Then for all distinct points $x, y \in T$ which are incompatible in the tree order of $T$, then $\odownward{x}\cap \odownward{y}$ separates $x$ and $y$ in $G$.
\end{lemma}

As already mentioned, an end is, intuitively, a ``point at infinity'' that can be reached by following a ray in a graph.
Ends were first formally defined by Halin in \cite{halin1964unendliche}.
We review the definition of an end and of the underlying topology of the end space below.
\begin{definition}[End space]\label{def:endspace}
    Let $G$ be a graph.
    \begin{enumerate}[label=(\alph*)]
        \item We say that two rays $R$ and $R'$ in $G$ are \emph{equivalent} if for every finite set $F\subseteq G$, the of $R$ and $R'$ in $G\setminus F$ lie in the same connected component of $G\setminus F$.
        
        \item The \emph{end space} of $G$, denoted by $\Omega(G)$, is the set of all equivalence classes of rays in $G$.
        An element $\epsilon$ of $\Omega(G)$ is called an \emph{end} of $G$.

        \item Let $F$ be a finite set and $C$ be a connected component of $G\setminus F$.
        We define $\Omega(C, F)$ as the set of all ends $\epsilon$ such that there exists a ray $R \in \epsilon$ contained in $C$.

        \item Given an end $\epsilon$ and a finite set $F$, we define $G(\epsilon, F)$ as the unique connected component of $G\setminus F$ which contains a tail of every ray in $\epsilon$.
        We define $\Omega(\epsilon, F)=\Omega(G(\epsilon, F), F)$.

        \item Given a finite set $F\subset \rmV(G)$, we define $\rmC(G, F) = \set{G(\varepsilon, F):\varepsilon\in \Omega(G)}$ (equivalently, $\rmC(G, F)$ is the family of connected components of $G\setminus F$ which are not rayless).

        \item We say that $U\subseteq \Omega(G)$ is \emph{open} if for every end $\epsilon$ in $U$, there exists a finite set $F$ such that $\Omega(\epsilon, F)\subseteq U$.
    \end{enumerate}
\end{definition}

It is easily seen that the collection of all open sets of $\Omega(G)$ forms a topology on $\Omega(G)$.
It is well known that this topology is Hausdorff, regular, Fréchet-Urysohn and zero-dimensional.
These spaces have also been shown to be hereditarily ultraparacompact \cite{kurkofka2021approximating} and monotonically normal \cite{pitz2023characterisingpathraybranch}.

\begin{definition}
    Let $G$ be a graph and $H\subseteq V(G)$.
    We define $\partial_{\Omega}(H)$ as the set of all ends $\epsilon \in \Omega(G)$ such that, for every finite set $F\subseteq H$, $G(\epsilon, F)\cap H\neq \emptyset$.

    We say that a subset $A\subseteq G$ is \emph{dispersed} if $\partial_{\Omega}(A)=\emptyset$.
\end{definition}
It is easily seen that $\partial_{\Omega}(H)$ is closed in $\Omega(G)$.
\subsubsection{On trees and rays}
Recall that a (order theoretic) tree is a partially ordered set $(T, <)$ such that for every $t \in T$, the set $\odownward{t}=\{s \in T: s< t\}$ is well-ordered by $<$.

\begin{definition}
    Let $T$ be a (order theoretic) tree.
    A subset $R\subset T$ is a \emph{ray} if it is a downward closed chain with no maximal element.
    Notice that $R$ needs not to be maximal.

    We say that $t\in T$ is a \emph{top} of a ray $R$ if $R = \odownward{t}$.
    Thus, $t \in T$ is a top of some ray if, and only if, the height of $t$ is a limit ordinal.

    The set of all rays of $T$ is denoted by $\mathcal{R}(T)$ and we topologize it as a subspace of $\mathcal P(T)\approx 2^T$, that is, a subbasis is given by the sets of the form $[t]=\{R \in \mathcal R(T): t \in R\}$ and $\mathcal R(T)\setminus [t]=\{R \in \mathcal R(T): t \notin R\}$ for $t \in T$.
\end{definition}

As shown in \cite{pitz2023characterisingpathraybranch}, we may describe a basis of the topology on $\mathcal{R}(T)$ as follows.

\begin{lemma}[{\cite[Lemma~2.1.]{pitz2023characterisingpathraybranch}}]
    Let $T$ be an order-theoretic tree and $x \in T$ be a ray.
    A local basis for $x$ is the collection of all sets of the form
    \begin{equation*}
        [t, F] := [t] \setminus \bigcup_{s \in F}[s]
    \end{equation*}
    where $t \in x$ and $F$ is a finite set of tops of $x$.
\end{lemma}

We will call the sets $[t, F]$, where $t \in T$ and $F$ is a (possibly empty) set of tops of a ray containing $t$, \emph{standard basics sets} of $\mathcal{R}(T)$.

An important recent representation theorem due to Kurkofka and Pitz states that every end space can be seen as a particular ray space of a rooted tree.
\begin{theorem}[\cite{kurkofka2024representationtheoremendspaces}]\label{theorem:representation}
    Let $X$ be a topological space. The following are equivalent:
    \begin{enumerate}[label=\emph{(\arabic*)}]
        \item $X$ is homeomorphic to the end space of a graph,
        \item $X$ is homeomorphic to the ray space of a special (rooted) order tree.
    \end{enumerate}
\end{theorem}

\subsection{Main results and structure of this paper}\label{subsec:mainresults}

In Section~2, we provide combinatorial characterizations of the Lindelöf degree and the extent of end spaces of graphs.
Namely:

\begin{theorem*}[\ref{theorem:lindelof}]
    Let $G$ be a graph.
    Then the end space $\Omega(G)$ is Lindelöf if, and only if, for every finite set $F\subseteq V(G)$, at most countable many components of $G\setminus F$ contains rays.
\end{theorem*}

\begin{theorem*}[\ref{theorem:lindelofDegree}]
    Let $G$ be a graph and $\kappa$ be an infinite cardinal.
    Then $L(\Omega(G))\leq \kappa$ if, and only if, for every finite set $F\subseteq V(G)$, at most $\kappa$ many components of $G\setminus F$ contains rays.
\end{theorem*}

We dedicate Section~3 to show the following results:

\begin{proposition*}[\ref{PROP_RaySpacesAreDSpaces}]
Every ray space is a $D$-space.
\end{proposition*}
\begin{corollary*}[\ref{PROP_EndSpacesAreDSpaces}]
Every end space is a $D$-space.
\end{corollary*}

In Section~4, we provide  combinatorial characterizations of the Rothberger property for end spaces of graphs and ray spaces of trees. 
Namely:

\begin{theorem*}[\ref{THM_RaySpaceRothberger}]
Let \( T \) be a rooted tree. The following are equivalent:
\begin{itemize}
    \item[(a)] \(\mathcal{R}(T)\) is Rothberger.
    \item[(b)] $\mathcal{R}(T)$ is Lindelöf and contains no topological copy of the Cantor space.
    \item[(c)] $\mathcal{R}(T)$ is Lindelöf and $T$ contains no subset which is order-isomorphic to the binary tree $2^{<\omega}$.
    \item[(d)] $\mathcal R(T)$ is Lindelöf and scattered.
\end{itemize}

Moreover, if $T$ is a pruned tree, then the ``Lindel\"of'' condition in all items may be swapped by ``every node has at most countably many successors''.
\end{theorem*}

\begin{corollary*}[\ref{COR_RothbergerEnds}]
Let $G$ be a graph. The following are equivalent:
\begin{enumerate}[label=(\alph*)]
    \item $\Omega(G)$ is Rothberger.

    \item $\Omega(G)$ is Lindel\"of and does not contain a copy of the Cantor space.
    
    \item $\Omega(G)$ is Lindel\"of and scattered.
    
    \item $\Omega(G)$ is Lindel\"of and $G$ contains no end-faithful subgraph which is a subdivision of the binary tree.

    \item For every $\subseteq$-increasing sequence $(F_n:n\in\omega)$ of finite subsets of $\mathrm{V}(G)$ there is a sequence $(C_n:n\in\omega)$, where each $C_n$ is a connected component of $G\setminus F_n$, so that every ray of $G$ has a tail in $C_n$ for some $n\in\omega$.

\end{enumerate}
\end{corollary*}

Using the results shown in the previous section, we present in Section 5 the following combinatorial characterizations of graphs whose end spaces have countably many ends: 

\begin{corollary*}[\ref{COR_CountableEnds_SeqFin}]
    A graph $G$ has countably many ends if, and only if, all the following conditions hold:
    \begin{enumerate}[label=(\roman*)]
        \item For every $\subseteq$-increasing sequence $(F_n:n\in\omega)$ of finite subsets of $\mathrm{V}(G)$ there is a sequence $(C_n:n\in\omega)$, where each $C_n$ is a connected component of $G\setminus F_n$, so that every ray of $G$ has a tail in $C_n$ for some $n\in\omega$.
        \item For every ray $R$ in $G$ there exists a $\subseteq$-increasing sequence $(F_n:n\in\omega)$ of finite subsets of $\mathrm{V}(G)$ such that, for every ray $R'$ in $G$ which is not equivalent to $R$, there exists an $n\in\omega$ such that $R$ and $R'$ lie in different connected components of $G\setminus F_n$.
    \end{enumerate}
\end{corollary*}

\begin{corollary*}[\ref{COR_CountableEnds_BinTree}]
    A graph $G$ has countably many ends if, and only if, all the following conditions hold:
    \begin{enumerate}[label=(\roman*)]
        \item $G$ contains no end-faithful subgraph which is a subdivision of the binary tree.
        \item For every ray $R$ in $G$ there exists a $\subseteq$-increasing sequence $(F_n:n\in\omega)$ of finite subsets of $\mathrm{V}(G)$ such that, for every ray $R'$ in $G$ which is not equivalent to $R$, there exists an $n\in\omega$ such that $R$ and $R'$ lie in different connected components of $G\setminus F_n$.
    \end{enumerate}
\end{corollary*}

In Section~6, we discuss some partial results regarding a combinatorial characterization of the Menger property.
In order to do so, we rely on the following notion:

\begin{definition*}[\ref{DEF_K-BaireRank}]
    Suppose $T$ is a rooted tree and $S\subseteq T$ is a subtree. We will say that $t\in T$ is \emph{$S$-compactly trivial} if every $s\in S$ above $t$ has finitely many successors in $S$.

    Moreover, we recursively define, for each ordinal $\alpha$, a subtree $\partial_K^\alpha(T)$ of $T$ as follows:
    \begin{itemize}
        \item $\partial_K^0(T)=T$;
        \item If $\partial_K^\alpha(T)$ has been defined, let $\partial_K^{\alpha+1}(T)$ be the subtree of $\partial_K^\alpha(T)$ obtained by removing all $\partial_K^\alpha(T)$-compactly trivial nodes from $\partial_K^\alpha(T)$;
        \item If $\beta$ is a limit ordinal and $\partial_K^\alpha(T)$ has been defined for all $\alpha<\beta$, let $\partial_K^\beta(T) = \bigcap_{\alpha<\beta} \partial_K^\alpha(T)$.
    \end{itemize}
\end{definition*}

With it, we thus obtain the following characterizations:

\begin{theorem*}[\ref{THM_RayMengerCharacterization}]
    For every rooted tree \( T \), the following are equivalent:
    \begin{enumerate}[label=(\alph*)]
        \item $\mR(T)$ is $\sigma$-compact.
        \item $\mathcal{R}(T)$ is Menger.
        \item $\mR(T)$ is Lindel\"of and there is no order embedding $\varphi\colon \omega^{<\omega}\to T$ such that for every $s\in \omega^{<\omega}$ there is a $t_s\in T$ with $t_s\geq\varphi(s)$ for which, for every $n\in\omega$, $\varphi(s^\smallfrown n)$ is a successor of $t_s$.
        \item $\mathcal{R}(T)$ is Lindel\"of and does not contain a closed topological copy of the Baire space $\omega^\omega$. 
        \item $\mathcal{R}(T)$ is Lindel\"of and, for every nonempty closed $X\subseteq \mR(T)$, there is an $x\in X$ with an open neighborhood $V_x$ which is compact in $X$.
        \item $\mathcal{R}(T)$ is Lindel\"of and there exists an ordinal $\gamma$ such that $\partial_K^\gamma(T)= \emptyset$.
    \end{enumerate}

    Moreover, if $T$ is a pruned tree, then the ``Lindel\"of'' condition in all items may be swapped by ``every node has at most countably many successors''.
\end{theorem*}

\begin{corollary*}[\ref{COR_EndMengerCharacterization}]
    If \( G \) is a graph, then the following are equivalent:
    
    \begin{enumerate}[label=(\alph*)]
        \item $\Omega(G)$ is $\sigma$-compact.
        \item $\Omega(G)$ is Menger.
        \item $\Omega(G)$ is Lindelöf and does not contain a closed topological copy of the Baire space $\omega^\omega$.
    \end{enumerate}
\end{corollary*}

In Section~7, we adapt our results to edge-end spaces of graphs.
More precisely, we show the following results:

\begin{corollary*}[\ref{COR_RothEdgeEnds}]
    Let $G$ be a graph. Then, the following are equivalent:
    \begin{enumerate}[label=(\alph*)]
        \item $\Omega_E(G)$ is Rothberger.

        \item $\Omega_E(G)$ is Lindelöf and does not contain a copy of the Cantor space.

        \item $\Omega_E(G)$ is Lindelöf and scattered.

        \item For every strictly $\subseteq$-increasing sequence $(F_n:n\in\omega)$ of finite subsets of $\mathrm{E}(G)$ there is a sequence $(C_n:n\in\omega)$, where each $C_n$ is a connected component of $G\setminus F_n$, so that every ray of $G$ has a tail in $C_n$ for some $n\in\omega$.
    \end{enumerate}
\end{corollary*}

\begin{corollary*}[\ref{COR_CountableEdgeEnds}]
    A graph $G$ has countably many edge-ends if, and only if, all the following conditions hold:
    \begin{enumerate}[label=(\roman*)]
        \item For every $\subseteq$-increasing sequence $(F_n:n\in\omega)$ of finite subsets of $\mathrm{E}(G)$ there is a sequence $(C_n:n\in\omega)$, where each $C_n$ is a connected component of $G\setminus F_n$, so that every ray of $G$ has a tail in $C_n$ for some $n\in\omega$.
        \item For every ray $R$ in $G$ there exists a $\subseteq$-increasing sequence $(F_n:n\in\omega)$ of finite subsets of $\mathrm{E}(G)$ such that, for every ray $R'$ in $G$ which is not edge-equivalent to $R$, there exists an $n\in\omega$ such that $R$ and $R'$ lie in different connected components of $G\setminus F_n$.
    \end{enumerate}
\end{corollary*}

\begin{corollary*}[\ref{COR_MengerEdgeEnds}]
    Let $G$ be a graph. Then, the following are equivalent:
    \begin{enumerate}[label=(\alph*)]
        \item $\Omega_E(G)$ is $\sigma$-compact.
        \item $\Omega_E(G)$ is Menger.
        \item $\Omega_E(G)$ is Lindelöf and does not contain a closed copy of the Baire space.
    \end{enumerate}
    Moreover, $\Omega_E(G)$ is a $D$-space.
\end{corollary*}

Moreover, by letting $\rmt(G)$ denote the set of vertices of $G$ which can be separated from any given $G$-ray by finitely many edges, we have also shown: 

\begin{corollary*}[\ref{COR_LindelofEdgeEnds}]
    Let $G$ be a graph. 
    Then, $L(\Omega_E(G))\le\kappa$ if, and only if, for every finite $F\subset \rmt(G)$, $|\rmC(G,F)|\le \kappa$.
\end{corollary*}

Finally, in Section~8, we conclude our paper with some final remarks and questions.
\section{On the Lindelöf property}

Recall that a topological space $X$ is \emph{compact} if every open cover of $X$ has a finite subcover.
Compactness is, perhaps, the most studied covering property in topology.
It plays a central role in analysis and topology, as many important theorems are related to compactness, such as the Heine–Borel theorem, Tychonoff's theorem, and the extreme value theorem.

A weaker property than compactness is the \emph{Lindelöf property}, which asserts that every open cover of $X$ admits a countable subcover.
In metrizable spaces, the Lindelöf property is equivalent to separability -- that is, the existence of a countable dense subset—as well as to the existence of a countable base for the topology.
However, for general topological spaces, and in particular for end spaces, these properties may differ and are not necessarily equivalent.

The following is a well-known combinatorial characterization of compactness for end spaces.

\begin{theorem}[{\cite[Corollary~4.4]{diestel2006end}}]
    For every graph $G$, its end space $\Omega(G)$ is compact if, and only if, for every finite set $F\subseteq V(G)$, only finitely many components of $G\setminus F$ contains rays.
\end{theorem}

We show, with the next theorem, that the Lindelöf property for $\Omega(G)$ is similarly governed by the number of components containing rays after removing finite sets of vertices.

\begin{theorem}\label{theorem:lindelof}
    Let $G$ be a graph.
    Then the end space $\Omega(G)$ is Lindelöf if, and only if, for every finite set $F\subseteq V(G)$, at most countable many components of $G\setminus F$ contains rays.
\end{theorem}

In this section, we actually prove a more general version of the result above.
Recall that, for a topological space $X$, the \emph{Lindelöf degree} $L(X)$ is the least infinite cardinal $\kappa$ such that every open cover of $X$ has a subcover of cardinality at most $\kappa$.
In particular, $X$ is Lindelöf if, and only if, $L(X)=\aleph_0$.
For background on the Lindelöf degree and other cardinal functions, see, for example, \cite{hodel1984cardinal}.

\begin{theorem}\label{theorem:lindelofDegree}
    Let $G$ be a graph and $\kappa$ be an infinite cardinal.
    Then $L(\Omega(G))\leq \kappa$ if, and only if, for every finite set $F\subseteq V(G)$, at most $\kappa$ many components of $G\setminus F$ contains rays.
\end{theorem}

Of course, Theorem~\ref{theorem:lindelofDegree} implies Theorem~\ref{theorem:lindelof} by putting $\kappa=\aleph_0$. 
To prove it, will need the following basic fact about normal trees, which is proved for the sake of completeness.
\begin{lemma}\label{lemma:normalSubtree}
    A subtree of a normal tree is normal.
\end{lemma}
\begin{proof}
    Let $T$ be a normal tree and $T'$ be a subtree of $T$.

    Let $x, y \in T'$ be distinct points incompatible in the tree order of $T'$.
    Then they are also incompatible in the tree order of $T$. Thus, $F=\odownward{x}\cap \odownward{y}\subseteq T'$ separates $x$ and $y$ in $G$.
    This implies that there is no $T'$-path $P$ from $x$ to $y$, as it would have to intersect $F\subseteq T'$.
\end{proof}
The following theorem is central in the theory of normal trees.
Recall that a subgraph of $G$ is \emph{rayless} if it does not contain any rays.
\begin{theorem}[\cite{jung1969wurzelbaume}]\label{theorem:jung}
    Let $G$ be a connected graph.
    Every rayless normal tree $T$ in $G$ is dispersed.
    Moreover, given a rayless normal tree $(T, r)$ in $G$ and a dispersed set $H\subseteq V(G)$, there exists a rayless normal tree $(T', r)$ in $G$ containing $T$ and $H$.
\end{theorem}
In particular, the previous theorem says that the dispersed sets are exactly the sets that are contained in a rayless normal tree.

We refine the previous result with the following observation, which will play a key role in the proof of Theorem~\ref{theorem:lindelofDegree}.

\begin{lemma}\label{lemma:normalTreeDispersed}
    Let $G$ be a graph and a dispersed set $H\subseteq V(G)$. Given a rayless normal tree $(T, r)$ in $G$ and a dispersed set $H\subseteq V(G)$, there exists a normal tree $(T', r)$ in $G$ such that:
    \begin{enumerate}[label=(\alph*)]
    \item $T, H\subseteq T'$.
    \item $(T', r)$ is minimal as a normal tree in the previous sense, that is, if $(T'', r)$ is a normal tree in $G$ such that $T, H\subseteq T''$ and such that $T''$ is a subgraph of $T'$, then $T''=T'$.
    \item If $H$ is finite, then $T'\setminus T$ is finite.
    \item If $H$ is infinite, then $|T'\setminus T|\leq|H|$.
    \end{enumerate}
\end{lemma}
\begin{proof}
    By \ref{theorem:jung}, there exists a rayless normal tree $(S, r)$ in $G$ containing $T$ and $H$.

    Let $T'=T\cup\bigcup_{x \in H} \odownward{x}$.
    $T'$ is easily seen to be a tree, so, by Lemma~\ref{lemma:normalSubtree}, it is normal. Therefore, (a) holds.
    Moreover, $T'\setminus T\subseteq \bigcup_{x \in H} \odownward{x}$, which is a union of $|H|$ many finite sets.
    Thus, (c) and (d) hold.

    Finally, (b) also holds: every subgraph of $T'$ which is a tree must contain all its vertices, as it must be downwards closed as it is a subtree, and it must also contain all its edges, as a tree becomes disconnected if any edge is removed.
\end{proof}

The following lemmas, though standard in the theory of normal trees, are included with proofs for completeness, as they will be instrumental in the main construction.

\begin{lemma}\label{lemma:normalTreeUnion}
    Let $G$ be a graph and $(T, r)$ be a normal tree in $G$.

    For every connected component $C$ of $G\setminus T$, let $(T_C, r)$ be a normal tree in $G$ extending $(T, r)$ such that $T_C\setminus T\subseteq C$.

    Let $T'=\bigcup\{T_C: C \text{ is a connected component of } G\setminus T\}$.
    Then $(T', r)$ is a normal tree in $G$.
    Moreover, if for every connected component $C$ of $G\setminus T$, $T_C$ is rayless, then $(T', r)$ is also rayless.
\end{lemma}

\begin{proof}
    Let $T'$ be the union of all $T_C$ as defined in the statement.
    
    $T'$ is easily seen to be a tree.
    We show it is normal. Fix $x, y$ incompatible in the tree order of $T'$.

    If $x, y$ are both in $T_C$ for some connected component $C$ of $G\setminus T$, then they are incompatible in the tree order of $T_C$, then there is no $T_C$-path from $x$ to $y$, and thus there is no $T'$-path from $x$ to $y$.

    Otherwise, let $C, D$ be connected components of $G\setminus T$ such that $x \in T_C$ and $y \in T_D$.
    As $x\notin T_D$ and $y \notin T_C$, then $x \in C$ and $y \in D$.
    Thus, there is no path from $x$ to $y$ in $G$ that does not intersect $T$, and therefore there is no $T'$-path from $x$ to $y$.

    It remains to show that if every $T_C$ is rayless, then so is $T'$.

    It suffices to show that $T'$ is dispersed, as no set containing rays is dispersed.
    Let $\epsilon$ be an end.
    There exists a finite set $F$ such that $G(\epsilon, F)$ is disjoint from $T$.
    Thus, $G(\epsilon, F)$ is contained in some connected component $C$ of $G\setminus T$.
    There exists a finite set $F'$ containing $F$ such that $G(\epsilon, F')$ is disjoint from $T_C$.

    Then $G(\epsilon, F')$ is disjoint from $T'$: indeed if it were to intersect $T'$, it would intersect some $T_D\setminus D$, where $D$ is a connected component of $G\setminus T$ distinct from $C$.
    But, as $G(\epsilon, F')$ is connected, that would imply that there is a path from $C$ to $D$ in $G\setminus T$, which is impossible.
\end{proof}

\begin{lemma}\label{lemma:normalRaylessNhood}
    Let $G$ be a graph and $(T, r)$ be a normal rayless tree in $G$.
    Then for every $C$ connected component of $G\setminus T$ there is a finite $F\subseteq T$ such that $C$ is a connected component of $G\setminus F$.
\end{lemma}
\begin{proof}
    Let $C$ be a connected component of $G\setminus T$.
    Let $F=\{x \in T: \exists y \in C\, \{x, y\}\in E(G)\}$.

    $F$ is a chain on $T$: given $x, x' \in F$, there exists $y, y' \in C\setminus T$ such that $\{x, y\}, \{x', y'\}\in E(G)$.
    As $y$ and $y'$ are vertices in the connected graph $C$, there exists a path $P$ contained in $C$ from $y$ to $y'$.
    Then $(x)^\frown P^\frown(x')$ is a $T$-path, therefore $x, x'$ are compatible.

    As $T$ is rayless, $F$ is finite, otherwise an increasing enumeration of $\bigcup_{x \in F} \odownward{x}$ would yield a ray in $T$.

    Now let $c \in C$ be any vertex and let $C'$ be the connected component of $G\setminus F$ containing $c$.
    $C$ is connected in $G\setminus F$ and contains $c$, so $C\subseteq C'$.

    To see that $C'\subseteq C$, assume this is not true.
    As $C'$ is connected, there exists $x \in C'\setminus C$ and $y \in C$ such that $\{x, y\}\in E(G)$.
    As $x \notin C$, which is a connected component of $G\setminus T$, and $y\in C$ is a neighbor of $x$, this means that $x \in T$. Thus, $x \in F$, a contradiction (since $C'$ is disjoint from $F$).

    Thus, $C'=C$, completing the proof.
\end{proof}

We now turn to the main lemma.
The argument is largely inspired by Theorem~1 of \cite{kurkofka2021approximating}; we adapt their approach, incorporating our additional hypothesis to control the cardinality of the constructed tree.

\begin{lemma}\label{lemma:mainLemma}
    Let $G$ be a connected graph  and $\kappa$ be an infinite cardinal such that for every finite set $F\subseteq V(G)$, at most $\kappa$ many components of $G\setminus F$ contains rays.

    Let $(F_\epsilon: \epsilon \in \Omega(G))$ be a family of finite sets.

    Then, for every $r \in G$, there exists a rayless normal tree $(T, r)$ in $G$ such that $|T|\leq \kappa$ and such that every connected component of $G\setminus T$ is contained in some $G(\epsilon, F_\epsilon)$.
\end{lemma}
To prove it, as it was done in \cite{kurkofka2021approximating}, we will rely on the following result.
\begin{theorem}[\cite{diestel2003graph}]\label{theorem:directions}Let $G$ be a graph and $d$ be a direction on $G$, that is, a mapping that associates to each finite subset $S$ of $V(G)$, a connected component of $C$ in a way that whenever $S\subseteq S'$ are finite subsets of $V(G)$, then $d(S')\subseteq d(S)$.

    Then there exists $\epsilon \in \Omega(G)$ such that for every finite set $F\subseteq V(G)$, $d(F)=G(\epsilon, F)$.
\end{theorem}
\begin{myproof}{Lemma}{lemma:mainLemma}
    Given $H\subseteq \rmV(G)$, we say that $H$ is \emph{bounded} if there exists $\epsilon\in \Omega(G)$ such that $H\subseteq G(\epsilon, F_\epsilon)$.
    Otherwise, we say that $H$ is \emph{unbounded}.

    Fix any $r \in G$.
    
    We recursively construct sequences $(T_n: n<\omega)$ and  $(S^n: n \in \omega)$ such that, for every $n \in \omega$:
    \begin{enumerate}[label=(\alph*)]
    \item $T_0=\{r\}$.
    \item $(T_{n+1}, r)$ is a rayless normal tree in $G$ extending $(T_n, r)$.
    \item $|T_{n}|\leq \kappa$.
    \item $S^n=(S^n_C\subseteq C: C \text{ is an unbounded connected component of } G\setminus T_n)$.
    \item $T_{n+1}=T_n\cup \bigcup\{T_{n+1}\cap C: C \text{ is an unbounded connected component of } G\setminus T_n\}$.
    \item For every unbounded connected component $C$ of $G\setminus T_n$, $S^n_C\subseteq C$ is a finite set such that $C\setminus S_C^n$ has:
        \begin{enumerate}[label=(\roman*)]
        \item no unbounded components, if there exists a finite set $S\subseteq C$ such that $C\setminus S$ has no unbounded components, or
        \item at least two unbounded components, otherwise.
        \end{enumerate}
    \item For every unbounded connected component $C$ of $G\setminus T_n$, $T_n\cup(C\cap T_{n+1})$ is a normal rayless tree in $G$ containing $T_n$ and $S^n_C$ which is minimal in this sense (see Lemma~\ref{lemma:normalTreeDispersed}).
    \end{enumerate}
    We show this is possible.
    \begin{proof}[Construction]
    Let $T_0=\{r\}$.
    Assume we have constructed $T_n$.
    We show how to construct $T_{n+1}$ and $S^n$.

    First, we claim that, given an unbounded connected component $C$ of $G\setminus T_n$, there exists a finite set $S\subseteq C$ such that $C\setminus S$ has at least two unbounded components or no unbounded component.
    Otherwise, let $d$ be the map that associated to each finite subset $S\subseteq C$ the unique unbounded component of $C\setminus S$.
    Then $d$ is a direction: if $S\subseteq S'$, then $d(S')\subseteq d(S)$, as $d(S')$ is an unbounded connected subset of $C\setminus S$ and therefore is contained in the unique unbounded component $d(S)$ of $C\setminus S$. By Theorem~\ref{theorem:directions}, there exists $\epsilon \in \Omega(G)$ such that for every finite set $F\subseteq V(G)$, $d(F)=C(\epsilon, F)$.
    But then $d(F_\epsilon\cap C)=C(\epsilon, F_\epsilon\cap C)\subseteq G(\epsilon, F_\epsilon)$, contradicting that $d(F_\epsilon\cap C)$ is unbounded.

    Now, for each unbounded connected component $C$ of $G\setminus T_n$, choose $S^n_C$ as in (f).
    As $S^n_C$ is finite, we apply Lemma~\ref{lemma:normalTreeDispersed} to obtain a normal rayless tree $(T_{n+1}^C, r)$ in $G$ extending $T_n$ and containing $S^n_C$ which is minimal in this sense, so that $T_{n+1}^C\setminus T_n$ is finite.

    By Lemma~\ref{lemma:normalTreeUnion}, $T_{n+1}=\bigcup\{T_{n+1}^C: C \text{ is an unbounded connected component of } G\setminus T_n\}$ is a normal rayless tree in $G$ extending $T_n$.
    Notice that $T_n\cup (T_{n+1}\cap C)=T_{n+1}^C$ for each unbounded connected component $C$ of $G\setminus T_n$, so clearly (e) and (g) holds.

    We must still verify that $|T_{n+1}|\leq \kappa$. As $|T_n|\leq \kappa$ and $T_{n+1}^C\setminus T_n$ is finite for every unbounded connected component $C$ of $G\setminus T_n$, it suffices to show that there are at most $\kappa$ many unbounded connected components of $G\setminus T_n$.
    By Lemma~\ref{lemma:normalRaylessNhood}, every such unbounded connected component is a connected component of $G\setminus F$ for some finite set $F\subseteq T_n$.
    Since $|T_n|\leq \kappa$, there are at most $\kappa$ many finite sets $F\subseteq T_n$, and, by hypothesis, there are at most $\kappa$ many unbounded connected components of $G\setminus F$ for every finite set $F\subseteq T_n$.
    Thus, as $\kappa$ is an infinite cardinal, there are at most $\kappa.\kappa=|\kappa\times \kappa|=\kappa$ many unbounded connected components of $G\setminus T_n$.
    
    This finishes the construction.
    \end{proof}
    Now, define $T=\bigcup_{n<\omega} T_n$.

    $(T, r)$ is a normal tree: $T$ has no cycles, as, since cycles are finite, there would be a cycle in some $T_n$.

    Moreover, $T$ is normal: given $x, y \in T$ incompatible in the tree order of $T$, there exists $n<\omega$ such that $x, y \in T_n$, and $x, y$ are incompatible in the tree order of $T_n$ as well.
    Thus, there is no $T_n$-path from $x$ to $y$, and therefore there is no $T$-path from $x$ to $y$.

    We now show that $T$ is rayless. Assume by contradiction that there exists a ray $R=(x_n: n<\omega)$ in $T$.
    We may assume that $x_0=r$.

    Recursively, construct $(n_k: k\in \omega)$ strictly increasing such that, for every $k \in \omega$, $x_{n_k}\notin T_{k}$.
    This is possible as $T_k$ is rayless.

    For each $k \in \omega$, $x_{n_k}$ lies in an unbounded connected component $C_k$ of $G\setminus T_k$.
    As $x_{n_{k+1}}$ lies in an unbounded connected component $C_{k+1}$ of $G\setminus T_{k+1}$, contained in $C_k$, it follows that $C_k\setminus S^{k}_{C_k}$ has at least two unbounded components.
    Let $D_k$ be an unbounded component of $C_k\setminus S^{k}_{C_k}$ distinct from $C_{k+1}$.
    Let $l_k$ be a neighbor of $D_k$ in $S^{k}_{C_k}$.
    As $l_k, x_{n_{k+1}}$ are points of $T_{k+1}\setminus T_k$ and as $\odownward{l_k}\cap \odownward{x_{n_{k+1}}}$ separates $l_k$ and $x_{n_{k+1}}$ in $G$ in case $l_k$ and $x_{n_{k+1}}$ are incompatible in the tree order of $T_{k+1}$, the unique path $P_k$ from $l_k$ to $x_{n_{k+1}}$ in $T_{k+1}$ is contained in $T_{k+1}\setminus T_k$.
    Thus, the $P_k$'s are pairwise disjoint.
    Moreover, the $D_k$'s are pairwise disjoint since if $k<k'$, then $D_{k'}$ is a subset of $C_k$, which is disjoint from $D_k$.

    Let $\epsilon$ be the end containing $R$ and let $F=F_\epsilon$.
    As $F$ is finite, there is $K$ such that for all $k\geq K$, $(D_k\cup P_k\cup\{x_{k}\})\cap F=\emptyset$.

    We claim that $D_{K+1}\subseteq G(\epsilon, F)$, contradicting the fact that it is unbounded.
    Indeed, first notice that $D_{K+1}$ is contained in the connected set $G(\epsilon, S^{K}_{C_K})$, and both $l_{K+1}$ and $x_{n_{K+1}}$ are inside of it.
    The path $P_{K+1}$ between them is in $T_{K+1}$, and the whole tail of $R$ starting at $x_{n_{K+1}}$ is contained in $G(\epsilon, F_\epsilon)$.
    As $F_{\epsilon}$ does not intersect the connected set $D_{K+1} \cup P_{K+1}$, we conclude that $D_{K+1}$ is bounded, a contradiction.

    This concludes that $T$ is rayless.
    It remains to seee that every component of $G\setminus T$ is bounded.
    Let $C$ be a connected component of $G\setminus T$.
    As $T$ is normal and rayless, the set $N$ of all neighbors of $C$ in $G$ (which must lie in $T$) is finite.
    Thus, there exists $n \in \omega$ such that $N\subseteq T_n$.
    Therefore, $C$ is a component of $G\setminus T_n$.
    As all unbounded components of $G\setminus T_n$ intersect $T_{n+1}$ and $C$ does not, it follows that $C$ is bounded.
\end{myproof}

At last, we have all the ingredients to prove the main result of this section:
\begin{myproof}{Theorem}{theorem:lindelofDegree}
    By adding a single vertex that connects to all other vertices, we may assume that $G$ is connected.
    \begin{description}
        \item[(a)$\Rightarrow$(b):] Assume that $L(\Omega(G))\leq \kappa$ and let $F\subseteq V(G)$ be a finite set.
        Then $(\Omega(C, F): C$ is a connected component of $G\setminus F$ which contains rays$)$ is an open partition of $\Omega(G)$, and therefore cannot have more than $L(\Omega(G))\leq \kappa$ many elements.

        \item[(b)$\Rightarrow$(a):] Let $\mathcal U$ be an open cover of $\Omega(G)$.
        For each $\epsilon \in \Omega(G)$, let $F_\epsilon$ be a finite set such that $\Omega(\epsilon, F_\epsilon)$ is contained in some $U_\epsilon \in \mathcal U$.
        By Lemma~\ref{lemma:mainLemma}, there exists a rayless normal tree $(T, r)$ in $G$ such that $|T|\leq \kappa$ and such that every connected component of $G\setminus T$ is contained in some $G(\epsilon, F_\epsilon)$.
        By Lemma~\ref{lemma:normalRaylessNhood}, every connected component $C$ of $G\setminus T$ is a connected component of $G\setminus F_C$ for some finite set $F_C\subseteq T$.

        As there are at most $\kappa$ many such $F_C\subseteq T$, and as each $F_C$ has at most $\kappa$ many connected components containing rays, it follows that there are at most $\kappa.\kappa=\kappa$ many connected components $C$ of $G\setminus T$ containing rays, and that for each such $C$, $\Omega(C, F_C)$ is open.
        Moreover, as $T$ is rayless and normal, every ray in $G$ must have a tail missing $T$, thus, the end it belongs to is in some $\Omega(C, F_C)$.
        Therefore, the set of all $\Omega(C, F_C)$ is an open refinement of $\mathcal U$ of cardinality at most $\kappa$.
    \end{description}
\end{myproof}

We have established yet another combinatorial characterization of the Lindelöf property for end spaces.
Since its statement is closely related to the Rothberger and Menger properties, we defer its discussion to Theorem~\ref{theorem:lindelofGraph2}, after those properties are introduced in the following sections.

As an application, we prove some facts about some other cardinal functions in the realm of end spaces.
Recall that the \emph{extent} of a topological space, denoted by $X$, and the \emph{cellularity} of a topological space, denoted by $c(X)$, are defined as follows:
\begin{align*}
    e(X) &= \sup\{|C|: C \text{ is a closed, discrete subset of } X\}+\aleph_0,\\
    c(X) &= \sup\{|\mathcal O|: \mathcal O \text{ is a family of pairwise disjoint open sets of } X\}+\aleph_0.
\end{align*}

\begin{proposition}\label{proposition:extentGraphs}
    Let $G$ be a graph.
    Then $e(\Omega(G))=L(\Omega(G))$ and $L(\Omega(G))\leq c(\Omega(G))$.
\end{proposition}
\begin{proof}
    It is a well-known and easy fact that, for every topological space, $e(X)\leq L(X)$:
    if $C$ is an infinite, discrete and closed subset of $X$, for each $c \in C$, let $\mathcal U_c$ be an open set containing $c$ such that $C\cap U_c=\{c\}$.
    Then $\{U_c: c \in \mathcal C\}\cup\{X\setminus C\}$ is an open cover of $X$ such that every subcover has cardinality at least $|C|$.

    Now we prove $L(\Omega(G))\leq e(\Omega(G))$ and $L(\Omega(G))\leq c(\Omega(G))$.
    Let $\kappa$ be a cardinal such that $\kappa<L(\Omega(G))$.
    By Theorem~\ref{theorem:lindelofDegree}, there exists a finite set $F\subseteq V(G)$ such that at least $\kappa^+$ many components $C$ of $G\setminus F$ contain rays.
    Let $\mathcal C$ be the collection of all such $C$.
    
    Then $\{\Omega(C, F): C \in \mathcal C\}$ is a family of pairwise disjoint open sets in $\Omega(G)$ of cardinality at least $\kappa^+$, thus, $c(\Omega(G))\geq \kappa^+$.
    Moreover, by selecting, for each $C \in \mathcal C$, a point $\epsilon_C \in \Omega(C, F)$, we obtain a closed discrete subset of $\Omega(G)$ of cardinality at least $\kappa^+$, thus, $e(\Omega(G))\geq \kappa^+$.

    We have proven that for all $\kappa<L(\Omega(G))$, $e(\Omega(G))\geq \kappa^+$ and $c(\Omega(G))\geq \kappa^+$, thus, $e(\Omega(G))\geq L(\Omega(G))$ and $c(\Omega(G))\geq L(\Omega(G))$, concluding the proof.
\end{proof}

In general, $e(X)=L(X)$ does not hold. For instance, the space $\omega_1$ of all countable ordinals is countably compact, therefore all infinite sets have accumulation points, and is not Lindelöf, therefore $e(\omega_1)=\aleph_0<\aleph_1=L(\omega_1)$.

Also, in general, $c(X)<L(X)$ is possible: fix, for each irrational number, a sequence of rational numbers converging to it.
Consider the topology where a basic open neighborhood of an irrational number is the number along with a tail of its fixed sequence, and every rational number is isolated.

Moreover, $L(X)=c(X)$ may fail for end-spaces: consider $X$ as the one point compactification of an uncountable set. Clearly, $c(X)$ is uncountable and $L(X)=\aleph_0$.
It is not difficult to see that $X$ is homeomorphic to an end space: consider for instance a graph $G$ comprising uncountably pairwise disjoint rays such that the set comprising the initial vertices of each ray forms a complete graph (see \myref{FIG_onePointCompactification}).

\begin{figure}[h!]
    \centering
    \begin{tikzpicture}
        \def\radius{0.15}
        \def\dy{0.75}
        \def\dx{1}
        \def\y{-2}
        \def\x{-4}
            \draw[thick] (\x-0.75,\y+4*\dy)  rectangle (\x+5*\dx,\y-1);
            \node (G) at (\x+4.65*\dx,\dy-0.25) {\large $G$};

            \draw[dashed] (\x-0.5,\y+0.75)  rectangle (\x+4.75*\dx,\y-0.75);
            \node (K) at (\x+4.4*\dx,\y-0.5) {$K_\kappa$};
            \foreach \n in {0,...,3}{
                \node (v\n) at (\x+\n*\dx,\y) [circle, fill, inner sep=1.5pt]{};
                } 
            \draw(v0) -- (v3);
            \foreach \j in {2,...,3}{ 
                        \draw (v0) edge[bend right=\j*15] (v\j);
                }
            \draw (v1) edge[bend left=30] (v3);
            \node (kdots) at (\x+3.75*\dx,\y)  {\tiny $\cdots$};
            
            \foreach \n in {0,...,3}{
                \draw[line width=2pt, white] (v\n) -- (\x+\n*\dx,\y+3*\dy);
                \draw[-{Latex[length=1.5mm]}] (v\n) -- (\x+\n*\dx,\y+3*\dy);
                } 
    \end{tikzpicture}
\caption{One point compactification of an uncountable discrete space as an end-space.}
\label{FIG_onePointCompactification}
\end{figure}

We have seen that $L(X)=e(X)$ holds for end spaces.
In fact, $L(X)=e(X)$ holds for all ray spaces.
Below, we extend \cite[Proposition~2.16]{kurkofka2024representationtheoremendspaces}.
As every end space is a ray space of a special tree, the next proposition is stronger than Proposition~\ref{proposition:extentGraphs}.
However, the presented proof of Proposition~\ref{proposition:extentGraphs} does not rely on this representation theorem, whose proof is highly nontrivial.
Thus, we have decided to present both proofs as some audiences may find one proof or the other more interesting.

\begin{proposition}\label{proposition:extentTrees}
    Let $T$ be a pruned rooted tree and $\kappa$ be an infinite cardinal.
    The following are equivalent:
    \begin{enumerate}[label=(\alph*)]
        \item $L(\mathcal R(T))=\kappa$.
        \item $e(\mathcal R(T))=\kappa$.
        \item $\kappa$ is the least infinite cardinal for which for every $t \in T$, $t$ has at most $\kappa$ many immediate successors in $T$.
    \end{enumerate}
\end{proposition}
\begin{proof}
    By \cite[Proposition~2.16]{kurkofka2024representationtheoremendspaces}, (a) and (c) are equivalent.

    We show that $L(\mathcal R(T))=e(\mathcal R(T))$.
    Again, it is well known that $L(\mathcal R(T))\geq e(\mathcal R(T))$.
    To see that $L(\mathcal R(T))\leq e(\mathcal R(T))$, let $\kappa$ be an infinite cardinal such that $\kappa<L(\mathcal R(T))$.
    Then there exists $t \in T$ such that $t$ has $\lambda\geq\kappa^+$ many immediate successors, $(t_\alpha: \alpha<\lambda)$.
    For each $\alpha<\lambda$, let $R_\alpha$ be a minimal ray such that $t_\alpha \in R_\alpha$.
    Then $\{R_\alpha: \alpha<\lambda\}$ is easily seen to be discrete.
    It is also closed: if $R$ is any ray distinct from the ones listed, either $t \in R$ or $t\notin R$.
    If $t \notin R$, then $R$ belongs to the open set $\{S \in \mathcal R(T): t \notin S\}$, which is disjoint from $\{R_\alpha: \alpha<\lambda\}$.
    If $t \in R$, then there exists a unique immediate successor $t_\beta$ of $t$ such that $t_\beta \in R$.
    As $R_\beta$ is minimal, $S\not\subseteq R_\beta$, so there exists $s \in S$ such that $s \notin R_\beta$.
    Then $\{S \in \mathcal R(T): s \in S\}$ is an open set containing $R$ and disjoint from $\{R_\alpha: \alpha<\lambda\}$.

    Thus, $e(\mathcal R(T))\geq \lambda\geq \kappa^+>\kappa$.
    This shows that $L(\mathcal R(T))\leq e(\mathcal R(T))$, completing the proof.
\end{proof}

We take the opportunity to prove a simple lemma which is somewhat related to the previous result we will need later.
Corollary~2.17 of \cite{pitz2023characterisingpathraybranch} states that a pruned, rooted tree $T$ has compact ray space if, and only if, every node of $T$ has finitely many immediate successors.
The proof is a consequence of Proposition~2.16 of \cite{pitz2023characterisingpathraybranch}, and the nontrivial direction of the proof relies on the ultraparacompactness of $\mathcal R(T)$.
Below, we show that the hypothesis of prunedness is not needed for this direction, and present a much simpler proof.

\begin{lemma}\label{compactFiniteSuccessors}
    Let $T$ be a rooted tree such that every node of $T$ has finitely many immediate successors.
    Then $\mathcal R(T)$ is compact.
\end{lemma}
\begin{proof}
    Consider $\mathcal P(T)$, the set of all subsets of $T$, with the topology induced by the product topology on $2^T$ -- thus, the sets of the form $\{a \in \mathcal P(T): t \in a\}$ and $\{a \in \mathcal P(T): t \notin a\}$, for $t \in T$, form a subbasis of $\mathcal P(T)$.
    As $\mathcal P(T)\approx 2^T$ is compact by Tychonoff's theorem, it suffices to show that $\mathcal R(T)$ is a closed subset of $\mathcal P(T)$.

    Let $r$ be the root of $T$.

    Let $b \in \mathcal P(T)\setminus \mathcal R(T)$.
    Then $b$ is not a ray.

    \textbf{Case 1:} $b=\emptyset$. Then $\{a \in \mP(T): r \notin a\}$ is a neighborhood of $b$ disjoint from $\mathcal R(T)$.

    \textbf{Case 2:} $b\neq\emptyset$ and is not closed downwards. Then there exists $t \in b$ and $s < t$ such that $s \notin b$.
    Then $\{a \in \mP(T): t \in a, s \notin a\}$ is a neighborhood of $b$ disjoint from $\mathcal R(T)$.

    \textbf{Case 3:} $b\neq \emptyset$ is closed downwards, but is not a chain.
    Then there exist $t, s \in b$ incomparable.
    Then $\{a \in \mP(T): t \in a, s \in a\}$ is a neighborhood of $b$ disjoint from $\mathcal R(T)$.

    \textbf{Case 4:} $b\neq \emptyset$ is closed downwards, a chain, but has a maximum element.
    Let $m$ be the maximum element of $b$.
    Let $F$ be the set of all immediate successors of $m$ in $T$, which is finite.
    Then 
    $$\{a \in \mP(T): m \in a, s \notin a \text{ for all } s \in F\}$$ 
    is a neighborhood of $b$ disjoint from $\mathcal R(T)$.

    This completes the proof.
\end{proof}

We finish this section with a simple application of Lemma~\ref{compactFiniteSuccessors} to prove another fact that will be needed later.
Given $X\subseteq \mR(T)$, let $T_X = \bigcup_{x\in X} x = \bigcup X$.
Clearly, $T_X$ is a subtree of $T$, and therefore $\mR(T_X)\subseteq \mR(T)$ and the ray space topology of $\mR(T_X)$ agrees with the subspace topology inherited from $\mR(T)$.
\begin{corollary}\label{LEMMA_Compact-DownClosure}
    If $T$ is a rooted tree and $K\subseteq\mR(T)$ is compact, $\mR(T_K)$ is compact.
\end{corollary}
\begin{proof}
    It suffices to see that every node of $T_K$ has only finitely many successors in $T_K$.
    Let $t \in T_K$. Then $K\cap [t]$ is a closed subset of $K$, hence, $K\cap [t]$ is compact.
    Let $F$ be the set of all successors of $t$ in $T_K$.
    Notice that $([s]\cap K: s\in F)$ is a partition of $K\cap [t]$ into pairwise disjoint nonempty clopen sets, and therefore must be finite.
    Thus, $F$ is finite.
\end{proof}
\section{Ray spaces are \texorpdfstring{$D$}{D}-spaces}

In \cite{pitz2023characterisingpathraybranch}, it has been shown that every ray space is paracompact -- that is, that every open cover has a locally finite refinement.
More specifically, the following result was proved:
\begin{proposition}[{\cite[Proposition~2.15.]{pitz2023characterisingpathraybranch}}]\label{LEMMA_Cover-Partition}
    Suppose $T$ is an order theoretic tree and $\mathcal{U}$ is an open cover of $\mathcal{R}(T)$.
    Then there is a partition $\mathcal{P}$ of $\mathcal{R}(T)$ comprised with standard basic sets refining $\mathcal{U}$ (that is, for every $A \in \mathcal P$ there exists $U \in \mathcal{U}$ such that $A \subseteq U$).
\end{proposition}
A related property is that of being a $D$-space.
A topological space $X$ is a \emph{$D$-space} if for every open neighborhood assignment $(\mathcal{U}_x:x\in X)$ (that is, for every $x\in X$, $\mathcal{U}_x$ is an open set containing $x$), there exists a closed discrete set $D\subseteq X$ such that $\bigcup_{x\in D} \mathcal{U}_x = X$.

The concept of $D$-spaces originated in the 1970s \cite{van1979some}, and has since become a central topic in set-theoretic topology.
For a comprehensive overview, see \cite{gruenhage2011survey}.

Despite significant progress, several fundamental questions about $D$-spaces remain unresolved. Notably, it is still an open problem whether every regular paracompact space is a $D$-space.
Even more surprisingly, whether every regular Lindelöf space is a $D$-space is also unanswered.
This longstanding problem appears as Problem 14 in \cite{hrusak2007twenty}.
Moreover, under the combinatorial principle $\diamondsuit$, there exists a hereditarily Lindelöf Hausdorff space that fails to be a $D$-space \cite{soukup2012counterexample}, highlighting the subtlety of the property.

In this section, we show that every ray space is a $D$-space, thereby identifying a new and natural class of paracompact spaces that satisfy this property.
In particular, this result implies that every end space of a graph is a $D$-space, further enriching the landscape of known examples.

\begin{proposition}\label{PROP_RaySpacesAreDSpaces}
Every ray space is a $D$-space.
\end{proposition}
\begin{proof}
    Let $(\mathcal{U}_x:x\in \mathcal{R}(T))$ be an open neighborhood assignment for $\mathcal{R}(T)$.
    We may assume that, for each $x \in \mathcal R(T)$, $\mathcal{U}_x$ is a standard basic open set of the form $[t_x,F_x]$.

    We recursively define $(e_\alpha:\alpha<\kappa^+)$, where $\kappa$ is the cardinality of $\mathcal{R}(T)$, as follows:
    \begin{enumerate}
        \item If $\bigcup_{\beta<\alpha} [t_{e_\beta},F_{e_\beta}] = \mathcal{R}(T)$, let $e_\alpha$ be arbitrary.
        \item Otherwise, let $e_\alpha$ be an arbitrary $\subseteq$-maximal ray of $\mathcal{R}(T)\setminus \bigcup_{\beta<\alpha} [t_{e_\beta},F_{e_\beta}]$ (which exists by Zorn's Lemma)
    \end{enumerate}
    Let $\delta<\kappa^+$ be the first ordinal such that $\bigcup_{\beta<\delta} [t_{e_\beta},F_{e_\beta}] = \mathcal{R}(T)$.
    
    Write $t_\beta=t_{e_\beta}$ and $F_\beta=F_{e_\beta}$ for each $\beta<\delta$.

    \begin{claim}
        For every $\beta<\alpha$, the following are complementary:
        \begin{itemize}
            \item $t_\alpha$ and $t_\beta$ are incompatible (in particular, $[t_{\beta},F_{\beta}]\cap [t_{\alpha},F_{\alpha}]= \emptyset$)
            \item $[t_{\beta},F_{\beta}]\subsetneq [t_{\alpha},F_{\alpha}]$ and $t_\beta \notin e_\alpha$ (in particular, $t_\alpha<t_\beta$).
        \end{itemize}
    \end{claim}
    \begin{proof}
            Suppose that $t_\alpha$ and $t_\beta$ are compatible, so that either $t_\beta\leq t_\alpha$ or $t_\alpha<t_\beta$.
            Since $e_\alpha\notin [t_{\beta},F_{\beta}]$, it follows that either $e_\alpha \cap F_\beta\neq \emptyset$ or $t_\beta \notin e_\alpha$.
            However, by the maximality of $e_\beta$, the first case cannot happen, therefore $t_\beta \notin e_\alpha$.
            This implies that $t_\alpha < t_\beta$.

            Now, given $x \in [t_{\beta},F_{\beta}]\cap [t_{\alpha},F_{\alpha}]$, we have $t_\alpha \in x$.
            If $x\cap F_\alpha\neq \emptyset$ then $x$ extends $e_\alpha$, which implies that $t_\beta\in x\setminus e_\alpha$, thus, $t_\beta$ is above some element of $F_\alpha$, implying that $[t_{\beta},F_{\beta}]\cap [t_{\alpha},F_{\alpha}]=\emptyset$, a contradiction.
    \end{proof}
    
    Now let 
    $$I=\{\xi<\delta: \forall \beta\in \delta\, \left([t_\beta, F_\beta]\cap [t_\xi,F_\xi]=\emptyset \;\vee\; [t_\beta, F_\beta]\subsetneq [t_\xi,F_\xi] \;\vee\;\xi=\beta\right)\}.$$

    We claim that $\{[t_\xi,F_\xi]:\xi\in I\}$ covers $\mathcal{R}(T)$.
    Indeed, given $x\in \mathcal{R}(T)$, let $\xi<\delta$ be the maximum such that $x\in [t_\xi,F_\xi]$.
    Such a $\xi$ exists as, otherwise, the above claim tells us that we would have an infinite decreasing sequence of elements of $T$.
    Again, by the claim, $\xi \in I$.

    Moreover, by the definition of $I$, if $\xi, \xi' \in I$ are distinct then $[t_\xi,F_\xi]\cap [t_{\xi'},F_{\xi'}]=\emptyset$.
    Then $\bigcup_{\xi\in I} [t_\xi,F_\xi]$ is a covering of $\mathcal{R}(T)$ by pairwise disjoint open sets.
    Finally, $\{e_\xi:\xi\in I\}$ is a closed discrete set as, for every $x \in \mathcal R(T)$, there exists $\xi\in I$ such that $x\in [t_\xi,F_\xi]$, and this open set intersects exactly one point of $\{e_\xi:\xi\in I\}$, namely $e_\xi$.
\end{proof}

It is well known and easy to verify that $e(X)=L(X)$ holds for $D$-spaces.
Thus, this gives us an alternate proof of Proposition~\ref{proposition:extentTrees}.

Furthermore, in view of the representation \myref{theorem:representation}:
\begin{corollary}\label{PROP_EndSpacesAreDSpaces}
Every end space is a $D$-space.
\end{corollary}
\section{On the Rothberger property and scattered ray spaces} \label{sec:rothberger}

While it is not known whether the Lindelöfness guarantees the $D$-space property, several stronger covering properties are known to do so.
The Menger property generalizes $\sigma$-compactness and is strictly stronger than Lindelöfness and, as shown in \cite{aurichi2010d}, it also implies the $D$-space property.
It was invented by Menger in \cite{menger2011einige}, who conjectured that, for metric spaces, it is equivalent to $\sigma$-compactness -- a once longstanding problem whose answer is now known to be false \cite{miller1988some}.

The Rothberger property was defined in \cite{rothberger1938verscharfung} in the context of the study of strengthenings of measure zero sets, and can be seen as a natural strengthening of the Menger property.
For metrizable spaces, it is equivalent to having strong measure zero with respect to every metric which induces its topology \cite{miller1988some}.

In this section, we investigate when ray spaces and end spaces of graphs are Rothberger, providing combinatorial and topological characterizations of this property.
We postpone the discussion on the Menger property to the next section.

\begin{definition}
    We say that a topological space $X$ is \emph{Rothberger} if for every sequence $(\mathcal{U}_n: n\in\omega)$ of open covers of $X$, there exists a sequence $(U_n: n\in\omega)$ such that $U_n\in \mathcal{U}_n$ for all $n\in\omega$ and $X=\bigcup_{n\in\omega} U_n$.

    We say that $X$ is \emph{Menger} if for every sequence $(\mathcal{U}_n: n\in\omega)$ of open covers of $X$, there exists a sequence $(\mathcal{F}_n: n\in\omega)$ such that each $\mathcal{F}_n\subset \mathcal{U}_n$ is finite and $X=\bigcup_{n\in\omega} \bigcup \mathcal{F}_n$.
\end{definition}

One readily sees that Rothberger spaces are Menger, and Menger spaces are Lindelöf.
The Rothberger property is actually much stronger than the Menger property: every compact space is easily seen to be Menger, however, the compact space $2^\omega$ is not Rothberger.

We start by showing the following characterization, which will be useful later.

\begin{theorem}\label{theorem:ScatteredRays}
    Let \( T \) be a rooted tree. The following are equivalent:
    \begin{itemize}
        \item[(a)] $\mathcal{R}(T)$ is scattered (that is, every non-empty subset of $\mathcal R(T)$ has an isolated point).
        \item[(b)] $\mathcal{R}(T)$ does not contain a copy of the Cantor space.
        \item[(c)] $T$ contains no subset which is order-isomorphic to the binary tree $2^{<\omega}$.
    \end{itemize}
\end{theorem}
\begin{proof}
    The proof that (a) implies (b) is straightforward.
    
    (b) implies (c): assume that $T$ has a subset $P$ order-isomorphic to the binary tree $2^{<\omega}$ and let $\phi: 2^{<\omega}\rightarrow T$ be an isomorphism.
    Given $f \in 2^{\omega}$, let $\psi(f)=\{r \in R: \exists n \in \omega\, r\leq \phi(f\restrict n)\}$.
    As $(\phi|n: n\in\omega)$ is a strictly increasing sequence in $T$, $\psi(f)$ is a ray in $T$, thus $\psi$ is well-defined.

    $\psi$ is injective: if $f\neq g$, let $n$ be the least such that $f(n)\neq g(n)$.
    Then $t_f=f|(n+1)$ and $t_g=g|(n+1)$ are incomparable in $2^{\omega}$, thus, $\phi(t_f)$ and $\phi(t_g)$ are incomparable in $T$.
    Notice that $\phi(t_f)\in \psi(f)$.
    However, $\phi(t_f)\notin \psi(g)$ as we would have $\phi(t_f)\leq \phi(g|m)$ for some $m\in\omega$, which entails that $t_f\subseteq g|m$, thus $n+1\leq m$ and $f(n)=t_f(n)=g(n)$, a contradiction.
    Notice that we have proved, in fact, that if $f\neq g$, then $\psi(f)\not \subseteq \psi(g)$.
    By a symmetrical argument, $\psi(g)\not \subseteq \psi(f)$.

    $\psi$ is continuous: given $f\in 2^{\omega}$, let $t \in \psi(f)$ and $F$ be a set of tops of $\psi(f)$.
    We must verify that there exists an open set $V\subseteq 2^{\omega}$ such that $f\in V$ and $\psi[V]\subseteq [t,F]$.
    As $t \in \psi(f)$, there exists $n \in \omega$ such that $t \leq \phi(f|n)$.
    Let $V=\{g \in 2^{\omega}: g|n=f|n\}$. Then $f \in V$ and $\psi[V]\subseteq [t,F]$, since if $g \in V$ and $g\neq f$, then $\psi(f)\not\subseteq \phi(g)$, so $\phi(g)\cap F=\emptyset$.

    As $\psi$ is continuous, injective and its domain is compact, it is a homeomorphism onto its range, completing the proof that (c) implies (b).
    \vspace{.3cm}
    
    Now we prove that (c) implies (a).
    For $S\subseteq T$, we will say that $s \in T$ is \emph{$S$-trivial} if $\{r \in S: r\geq s\}$ is a chain.
    
    As $T$ contains no subset order-isomorphic to $2^{<\omega}$, for every nonempty subtree $S$ of $T$ and $t \in S$ there exists an $S$-trivial $s\geq t$ (otherwise we may recursively construct a copy of $2^{<\omega}$ in $S\subseteq T$).

    Given a subtree $S$ of $T$, let $P(S)=\{s \in S: s \text{ is not $S$-trivial}\}$.

    Recursively define, for $\alpha<|T|^+$, $T_0=T$, $T_\alpha=P(T_\beta)$ if $\alpha=\beta+1$ and $T_\alpha=\bigcap_{\beta<\alpha} T_\beta$ if $\alpha$ is a limit ordinal.

    There must exist (the smallest) $\delta<|T|^+$ such that $T_\delta=T_{\delta+1}$.
    $T_\delta$ has no $T_\delta$-trivial element. Thus, $T_\delta=\emptyset$.

    For each $\beta\leq \delta$, let $V_\beta=\bigcup\{[s]: s \in T\setminus T_\beta\}$.
    $(V_\beta: \beta\leq \delta)$ is clearly an increasing family of open subsets of $\mathcal R(T)$ such that $V_0=\emptyset$, $V_\delta=\mathcal R(T)$ and $V_\xi=\bigcup_{\beta<\xi} V_\beta$ if $\xi$ is a limit ordinal.

    Now, given a non-empty $X\subseteq \mathcal R(T)$, let $\beta$ be the smallest ordinal such that $X\cap V_\beta\neq \emptyset$ and fix a minimal ray $x \in X\cap V_\beta$.
    Notice that $\beta$ must be a successor ordinal, say $\beta=\alpha+1$.
    We claim that $x$ is isolated in $X$.
    Indeed, let $s \in T_\alpha\setminus T_{\alpha+1}$ be such that $x \in [s]$.
    Let $F$ be the set of all tops of $x$ in $T_\alpha$.
    Then $|F|\leq 1$.

    We claim that $[s, F]\cap X=\{x\}$.
    Indeed, if $y \in [s, F]\cap X$, then $s \in y$.

    If $x, y$ are not contained in one another, then they give us two incompatible points $t_x \in x$, $t_y \in y$.
    As $s$ is $T_\beta$-trivial, one of $t_x, t_y$ is in $V_\xi$ for some $\xi<\beta$, which shows that $V_\xi\cap X\neq \emptyset$, contradicting the minimality of $\beta$.

    Thus, $x, y$ are contained in one another.
    By the minimality of $x$, $x\subseteq y$.
    If $y$ contains a top $t_y$ of $x$ which is not in $F$, then, again, $y$ violates the minimality of $\beta$.
    Thus, as $y \in [s, F]$, $y=x$.
\end{proof}

Now, let us consider two technical lemmas:

\begin{lemma}[Particular case of Corollary 5.5 in \cite{koloschin2023end}]\label{lemma:compactComponents}
    Let $G$ be a graph and $F$ be a closed subspace of $\Omega(G)$.
    Then there exists an induced subgraph $H$ of $G$ such that the natural inclusion mapping $i:\Omega(H)\to \Omega(G)$ is a topological embedding with image $F$.
\end{lemma}

\begin{lemma}[Particular case of Theorem 4.2 in \cite{kurkofka2022countablydeterminedends}]\label{lemma:SecCountableNormalTree}
    For a graph $G$, $\Omega(G)$ is second countable if, and only if, there exists a countable normal tree $T\subseteq G$ such that the natural inclusion mapping $i:\Omega(T)\to \Omega(G)$ is a topological homeomorphism.
\end{lemma}

Recall that a subgraph $H$ of a graph $G$ is \emph{end-faithful} if two $H$-rays are equivalent in $H$ if, and only if, they are equivalent in $G$.
Furthermore, recall that a graph $H$ is a \emph{subdivision} of a graph $G$ if $H$ can be obtained from $G$ by replacing edges of $G$ by internally vertex-disjoint paths.

With that, we thus obtain:

\begin{corollary}\label{corollary:scatteredends}
    Let $G$ be a graph. Then the following are equivalent:
    
    \begin{enumerate}[label=(\alph*)]
        \item\label{item_ScatteredEnds} $\Omega(G)$ is scattered.
        \item\label{item_NoCantorEnds} $\Omega(G)$ contains no topological copy of the Cantor space $2^\omega$.
        \item\label{item_NoEndFaithBinaryTree} $G$ contains no end-faithful subgraph which is a subdivision of the binary tree.
    \end{enumerate}
\end{corollary}
\begin{proof}
    Equivalence between \ref{item_ScatteredEnds} and \ref{item_NoCantorEnds} follows directly from Theorem~\ref{theorem:ScatteredRays} and Theorem~\ref{theorem:representation}. 

    It is clear that \ref{item_NoCantorEnds} implies \ref{item_NoEndFaithBinaryTree}, since the set of ends of an end-faithful subdivision of the binary tree contained in $G$ must necessarily be, as a subspace of $\Omega(G)$, homeomorphic to the Cantor space.
    
    So, striving to prove that not \ref{item_NoCantorEnds} implies not \ref{item_NoEndFaithBinaryTree},  suppose that $\Omega(G)$ contains a topological copy of the Cantor space, say $K\subseteq \Omega(G)$.

    By Lemma~\ref{lemma:compactComponents}, there exists an induced subgraph $H$ of $G$ such that the natural inclusion mapping $i:\Omega(H)\to \Omega(G)$ is a topological embedding with image $K$.
    Since $\Omega(H)$ is homeomorphic to the Cantor space, which is second countable, it follows from \myref{lemma:SecCountableNormalTree} that there exists a countable normal tree $T\subseteq H$ such that the natural inclusion mapping $i:\Omega(T)\to \Omega(H)$ is a topological homeomorphism.

    Thus, since $\Omega(T)$ is homeomorphic to the Cantor space, it follows that $T$ contains a subdivision of the binary tree, which concludes the proof.
\end{proof}

We can prove the following technical lemmas, which will be useful for our main result:

\begin{lemma}\label{LEMMA_SplitingPartition}
Let $P=([t_i, F_i]: i\in I)$ be a partition of $\mathcal R(T)$ into standard basic sets and $\mathcal U$ be an open cover of $\mathcal R(T)$.

Then there exists a partition $P' = ([s_j, G_j]: j\in J)$ of $\mathcal R(T)$ into standard basic sets refining $P$ and $\mathcal U$ such that for every $j\in J$, the unique $i \in I$ with $[s_j, G_j]\subseteq [t_i, F_i]$ satisfies that $t_i<s_j$.
\end{lemma}
\begin{proof}

    First, let 
    $$\mathcal U' = \set{U\cap V: U\in \mathcal U \text{ and } V\in P}$$ 
    and note that $\mU'$ is an open cover of $\mathcal R(T)$ refining both $P$ and $\mathcal U$.

    By Proposition~\ref{LEMMA_Cover-Partition}, there exists a partition $Q=([r_k, H_k]: k\in K)$ of $\mathcal R(T)$ into standard basic sets refining $\mathcal U'$.
    Now fix $k \in K$.
    Let $R_k \in \mathcal R(T)$ be a ray containing $r_k$ such that every member of $H_k$ is a top of $R_k$ (such ray is uniquely determined if $H_k\neq \emptyset$, otherwise we may pick any ray containing $r_k$).
    Also, let $i_k\in I$ be such that $[r_k, H_k]\subseteq [t_{i_k}, F_{i_k}]$.

    For our fixed $k$, we now define $r_k^*$ as follows.
    If $r_k\leq t_{i_k}$, notice that $[r_k, H_k]=[t_{i_k}, H_k]$ as $[r_k, H_k]\subseteq [t_{i_k}, F_{i_k}]$.
    In this case, let $r_k^*=t_{i_k}$.
    Otherwise, let $r_k^*=r_k$.
    In any case, let $(s_j: j \in J_k)$ be an enumeration of all direct successors of $r_k^*$ in $T$.
    Here, we assume that $(J_k: k\in K)$ is a family of pairwise disjoint sets of indices.

    Given $j \in J_k$, if $s_j \in R$, let $G_j=H_k$.
    Otherwise, let $G_j=\emptyset$.
    Thus, $P'_k=([s_j, G_j]: j\in J_k)$ is a partition of $[r_k, H_k]$ into standard basic sets refining $[r_k, H_k]$.

    Thus, $P'=\bigcup_{k \in K} P'_k$ is a partition of $\mathcal R(T)$ into standard basic sets refining $Q$, and therefore $P$ and $\mathcal U$.

    Moreover, given $j \in J_k$, $r_j>r_k^*\geq t_{i_k}$, so $t_{i_k}<s_j$, and $[s_j, G_j]\subseteq [r_k^*, H_k]\subseteq [r_k, H_k]\subseteq [t_{i_k}, F_{i_k}]$.

    Then, the family $P'$ is as required.
\end{proof}

\begin{lemma}\label{lemma:split}
    Let $X$ be a topological space and $(\mathcal{U}_n: n\in A)$, where $A$ is countable, be a sequence of countable open covers of $X$ which shows that $X$ is not Rothberger -- that is, for every sequence $(U_n: n\in A)$ such that $U_n\in \mathcal{U}_n$ for all $n\in A$, there exists $x\in X$ such that $x\notin \bigcup_{n\in A} U_n$.

    Then for every $N \in A$ there exists an infinite $B\subseteq A$ and distinct $U, V \in \mathcal{U}_N$ such that ($\mathcal{U}_n: n\in B$) shows that both $U$ and $V$ are not Rothberger.
\end{lemma}
\begin{proof}
    First, assume that for no $U \in \mathcal U_N$ there is an infinite $B\subseteq A$ such that $(\mathcal U_n: n\in B)$ shows that $U$ is not Rothberger.
    
    Let $(B_k: k \in \omega)$ be a partition of $A$ and write $\mathcal U_N=\{U_k: k \in \omega\}$.
    By assumption, for each $k \in \omega$ there exists a sequence $(V_n^k: n \in B_k)$ such that $V_n^k\in \mathcal U_n$ for all $n\in B_k$ and $U_k\subseteq \bigcup_{n\in B_k} V_n^k$.
    Then $\bigcup_{k\in\omega} \bigcup_{n\in B_k} V_n^k = \bigcup_{k \in \omega} U_k$, a contradiction.

    Now assume that for every $B\subseteq A$ and for every $U, V \in \mathcal U_N$, $(\mathcal U_n: n\in B)$ does not show that both $U, V$ are not Rothberger.
    By the previous argument, there exists an infinite $B\subseteq A$ such that $(\mathcal U_n: n\in B)$ shows that some $U \in \mathcal U_N$ is not Rothberger.
    Fix $U$ and $B$.
    Of course, $N\notin B$.

    Let $(B_k: k \in \omega)$ be a partition of $B$ into infinitely many infinite sets and enumerate $\mathcal U_N=\setminus\{U\}=\{U_k: k \in \omega\}$.
    By assumption, for each $k \in \omega$ there exists a sequence $(V_n^k: n \in B_k)$ such that $V_n^k\in \mathcal U_n$ for all $n\in B_k$ and $U_k\subseteq \bigcup_{n\in B_k} V_n^k$. But then $U \cup \bigcup_{k \in \omega} \bigcup_{n\in B_k} V_n^k \supseteq U \cup \bigcup_{k \in \omega} U_k =X$, a contradiction as now $(\mathcal U_n: n\in A)$ does not show that $X$ is not Rothberger.
\end{proof}

The following theorem encapsulates these relationships and highlights the interplay between Lindelöfness, scatteredness, and the absence of Cantor subspaces.

\begin{theorem}\label{THM_RaySpaceRothberger}
Let \( T \) be a rooted tree. The following are equivalent:
\begin{itemize}
    \item[(a)] \(\mathcal{R}(T)\) is Rothberger.
    \item[(b)] $\mathcal{R}(T)$ is Lindelöf and contains no topological copy of the Cantor space.
    \item[(c)] $\mathcal{R}(T)$ is Lindelöf and $T$ contains no subset which is order-isomorphic to the binary tree $2^{<\omega}$.
    \item[(d)] $\mathcal R(T)$ is Lindelöf and scattered.
\end{itemize}

Moreover, if $T$ is a pruned tree, then the ``Lindel\"of'' condition in all items may be swapped by ``every node has at most countably many successors''.
\end{theorem}
\begin{proof}

    (a) implies (b):
    If $\mathcal R(T)$ contains a subspace $\mathcal C$ homeomorphic to the Cantor space, then $\mathcal C$ is closed, as the Cantor space is compact and $\mathcal R(T)$ is Hausdorff.
    Thus, as the Rothberger property is hereditary for closed subspaces, $\mathcal R(T)$ is not Rothberger for the Cantor space is not Rothberger.
    \vspace{.3cm}

    (c) implies (a).
    We proceed by the contrapositive.

    Suppose that \(\mathcal{R}(T)\) is Lindel\"of, but not Rothberger, so that we may fix a sequence $S = \seq{\mathcal{U}_n:n\in\omega}$ of open covers attesting this fact. Since $\mathcal{R}(T)$ is Lindelöf, we can assume, by recursively applying \myref{LEMMA_SplitingPartition}, that:
    
    \begin{enumerate}[label=(\alph*)]
        \item Each $\mathcal{U}_n$ is a countable partition comprised of standard basic sets indexed as $([t_i, F_i]: i \in I_n)$, where the $I_n$'s are pairwise disjoint sets of indices.
        \item Each $\mathcal{U}_{n+1}$ is a refinement of $\mathcal{U}_n$ and for all $i \in I_n$, $J \in I_{n+1}$, if $[t_j, F_j]\subseteq [t_i, F_i]$, then $t_i < t_j$.
    \end{enumerate}
We will recursively define, for each $s\in 2^{<\omega}$, a natural number $N_s$, an infinite $A_s\subseteq \omega$ and $U_s\in \mathcal{U}_{\min A_s}$ such that:
    \begin{enumerate}[label=(\roman*)]
        \item $N_s=0$,
        \item $N_s<N_t$ if $s\subsetneq t$,
        \item $A_t\subseteq A_s$ if $s\subseteq t$,
        \item $U_{s}\in \mathcal{U}_{N_s}$,
        \item $U_{s^\frown (0)}$ and $U_{s^\frown (1)}$ are distinct and disjoint,
        \item $(\mathcal U_n: n\in A_s)$ attests that $U_{s}$ and $U_{s^\frown(1)}$ are not Rothberger.
    \end{enumerate}
    
    To do that, we proceed as follows.
    Let $N_s=0$. By Lemma \ref{lemma:split}, there exists $U_\emptyset \in \mathcal U_0$ and an infinite $A_\emptyset \subseteq \omega$ such that $(\mathcal U_n: n\in A_\emptyset)$ attests that $U_\emptyset$ is not Rothberger.

    Suppose we have defined $U_s$, $A_s$ for some $s\in 2^{<\omega}$.
    By Lemma \ref{lemma:split} applied to any $M \in A_s$ with $M>N_s$, there exist distinct $U_{s^\frown(0)}, U_{s^\frown(1)} \in \mathcal U_M$ and  infinite $A_{s^\frown(0)}=A_{s^\frown(1)}\subseteq A_s$ such that $(\mathcal U_n: n\in A_{s^\frown(i)})$ attests that both $U_{s^\frown(0)}, U_{s^\frown(1)}$ are not Rothberger.
    Let $N_{s^\frown(i)}=M$ for $i\in\{0,1\}$.
    \vspace{.2cm}

    For each $s\in 2^{<\omega}$, let $i_s$ be such that $U_s=[t_{i_s}, F_{i_s}]$ and write $t_s=t_{i_s}$.
    Thus, by (b) above, $t_s < t_{s^\frown(i)}$ for $i\in\{0,1\}$, and as in the proof of (c) implies (b), $t_{s^\frown(0)}$ and $t_{s^\frown(1)}$ are incomparable and for every $f \in 2^\omega$, the set $R_f=\{s\in T: \exists n \in \omega\, s\leq t_{f|n}\}$ is a ray in $T$ and $R_f\in \bigcap_{n\in\omega} U_{f|n}$.
    Thus, if $f\neq g$, then $R_f\neq R_g$.

    We now recursively construct an order embedding $\theta:2^{<\omega}\rightarrow 2^{<\omega}$ so that:
    \begin{enumerate}[label=(\roman*)]
        \item $\theta(\emptyset)=\emptyset$,
        \item $\theta(s)\subsetneq \theta(t)$ if, and only if $s\subsetneq t$,
        \item $t_{\theta(s)}< t_{\theta(t)}$ if, and only if $s\subsetneq t$.
    \end{enumerate}
    
    Assume we have defined $\theta(s)$ for some $s\in 2^{<\omega}$.
    We need to find two incomparable $s_0, s_1\in 2^{<\omega}$ such that $\theta(s)< s_i$, $t_{\theta(s)}< t_{s_i}$ for $i\in\{0,1\}$, and such that $t_{\theta(s_0)}$ and $t_{\theta(s_1)}$ are incomparable.
    With that done, we may let $\theta(s^\frown(i))=s_i$ for $i\in\{0,1\}$.
    Thus, the recursion will be complete, and we may finish the proof by defining $\phi(s)=t_{\theta(s)}$ for all $s\in 2^{<\omega}$.

    Assume that there are no such $s_0, s_1$.
    Then $C=\{t_{s'}: s'\in 2^{<\omega}, s'\geq \theta(s)\}$ is a chain in $T$.
    Let $R=\{t \in T: \exists s'\in 2^{<\omega}, s'\geq \theta(s), t\leq t_{s'}\}$, that is, let $R$ be the smallest ray containing $C$.
    Notice that $R=\bigcup\{ R_f: \theta(s)\subseteq f\in 2^\omega\}$.

    As $R$ is well-ordered and each $R_f$ is a different initial segment (possibly equal to $R$), write $(f_\alpha: \alpha<\delta)$ an enumeration of $\{f\in 2^\omega: \theta(s)\subseteq f\}$ such that $R_{f_\alpha}\subsetneq R_{f_\beta}$ if $\alpha<\beta<\delta$.
    Thus, $\delta$ is some ordinal greater or equal to $\mathfrak c$.

    For each ordinal $\alpha<\delta$, $\{t_{u}: u\in 2^{<\omega}, \subseteq u\subseteq f_\alpha\}$ is unbounded in $R_{f_\alpha}$, so we may fix $u_\alpha \in 2^{<\omega}$ such that $t_{u_\alpha}\in R_{f_{\alpha+1}}\setminus R_{f_\alpha}$.
    Thus, if $\alpha<\beta<\delta$, then $t_{u_\alpha}$ and $t_{u_\beta}$ are distinct, and therefore $u_\alpha\neq u_\beta$.
    Thus, the mapping $\alpha\mapsto u_\alpha$ is injective, so $\delta\leq \omega$, a contradiction.

    Finally, (c) is equivalent to (d) by \myref{theorem:ScatteredRays}.
\end{proof}

We should emphasize that there are scattered end spaces which are not Rothberger, for instance, consider an uncountable discrete space.
Such space is not Rothberger as it is not Lindelöf, and it can be attained as the end space of a star with uncountably many rays.

Now, turning our attention back to end spaces, we have the following combinatorial and topological characterizations.
\begin{corollary}\label{COR_RothbergerEnds}
Let $G$ be a graph. The following are equivalent:
\begin{enumerate}[label=(\alph*)]
    \item $\Omega(G)$ is Rothberger.

    \item $\Omega(G)$ is Lindel\"of and does not contain a copy of the Cantor space.
    
    \item $\Omega(G)$ is Lindel\"of and scattered.
    
    \item $\Omega(G)$ is Lindel\"of and $G$ contains no end-faithful subgraph which is a subdivision of the binary tree.

    \item For every $\subseteq$-increasing sequence $(F_n:n\in\omega)$ of finite subsets of $\mathrm{V}(G)$ there is a sequence $(C_n:n\in\omega)$, where each $C_n$ is a connected component of $G\setminus F_n$, so that every ray of $G$ has a tail in $C_n$ for some $n\in\omega$.

\end{enumerate}
\end{corollary}
\begin{proof}
    It is clear that (a) implies (e), while equivalence between (b), (c) and (d) follows directly from Corollary~\ref{corollary:scatteredends}.
    Furthermore, since every end space is homeomorphic to the ray space of a special rooted tree (Theorem~\ref{theorem:representation} in \cite{pitz2023characterisingpathraybranch}), equivalence between (a), (b) follows directly from Theorem~\ref{THM_RaySpaceRothberger}.

    Thus, we only need to show that (e) implies (b) (we do so by its contrapositive).

    In this proof, we denote the constant sequence $(0, 0, \dots)$ by $\mathbf 0$ and the constant sequence $(1, 1, \dots)$ by $\mathbf 1$.
    Moreover, if $n \in \omega$, let $[0]_n=(0, \ldots, 0)\in 2^n$ and $[1]_n=(1, \ldots, 1)\in 2^n$.
    Finally, if $s_0, \dots, s_k \in 2^{<\omega}$, we write $\bigsqcup_{i\leq k} s_i$ for the concatenation of the sequences $s_0^\frown \dots ^\frown s_k$, and $\bigsqcup_{i<0} s_i$ is the empty sequence.
    \vspace{.2cm}

    It is clear, in view of \myref{theorem:lindelof}, that (e) fails if $\Omega(G)$ is not Lindelöf.
    On the other hand, suppose $\psi\colon 2^\omega\to \Omega(G)$ is a topological embedding.
    We will show that (e) fails.

    We recursively construct, for $k \in \omega$, $F_k$, $n_k$, for all $k \in \omega$:
    \begin{enumerate}[label=(\roman*)]
        \item $F_k\subsetneq F_{k+1}$.
        \item $n_k<n_{k+1}$.
        \item For $s \in 2^{k}$, $\Omega_G\left(\psi\left(\bigsqcup_{i<k} [s(i)]_{n_k}^\frown \mathbf 0\right), F_k\right)\cap \Omega_G\left(\psi\left(\bigsqcup_{i<k} [s(i)]_{n_k}^\frown \mathbf 1\right), F_k\right)=\emptyset$.
        \item For $m<2$, $\left\{f \in 2^\omega: \bigsqcup_{i<k+1} [f(i)]_{n_k}\subseteq f\right\}\subseteq \psi^{-1}\left[\Omega_G\left(\psi\left(\bigsqcup_{i<k+1} [f(i)]_{n_k}^\frown \mathbf m\right), F_k\right)\right]$.
    \end{enumerate}

    This is easily achieved by the injectivity and continuity of $\psi$.
    Then $(F_k : k\in\omega)$ is an increasing sequence of finite subsets of $\mathrm{V}(G)$ as desired.

    Indeed, let $(C_n:n\in\omega)$ be such that each $C_n$ is a connected component of $G\setminus F_n$.
    Recursively, define $f|(n_0+\dots+n_k)$ so that the connected component of $G\setminus F_k$ containing rays of $\psi\left[\left\{f \in 2^\omega: \bigsqcup_{i<k+1} [f(i)]_{n_k}\subseteq f\right\}\right]$ is not $C_k$.
    Then $\psi(f)$ has no tail in $C_k$ for all $k\in\omega$.
\end{proof}

\section{Countable ray and end spaces}
In this section, we characterize the countability of ray and end spaces in terms of combinatorial properties of the underlying graph or tree, relating it to the Rothberger property and its combinatorial characterization developed in the previous section.

Obviously, for countable ray spaces all points are $G_\delta$ -- something that happens for every $T_1$ countable topological space -- and every point has a countable local neighborhood -- which fails in general.
We take the opportunity to show that, in such spaces, the \emph{character} and the \emph{pseudocharacter} of each point coincide.

Recall that, given a $T_1$ topological space $X$ and a point $x\in X$, the \emph{character} of $x$ in $X$, denoted by $\chi(x,X)$, is the minimum cardinality of a local base for $x$ in $X$, and the pseudocharacter of $x$ in $X$, denoted by $\psi(x,X)$, is the minimum cardinality of a local \emph{pseudobase} for $x$ in $X$, i.e., a family $\mU$ of open sets in $X$ such that $\bigcap\mU=\set{x}$.
Then the \emph{character} of $X$ is defined as $\chi(X)=\sup\set{\chi(x,X):x\in X}$, and the \emph{pseudocharacter} of $X$ is defined as $\psi(X)=\sup\set{\psi(x,X):x\in X}$.
Of course, for every $x\in X$, we have that $\psi(x,X)\leq \chi(x,X)$, and thus $\psi(X)\leq \chi(X)$.

\begin{theorem}\label{THM_CharPseudoCharRays}
    For every rooted tree $T$ and $x \in \mR(T)$, we have that $\chi(x,\mR(T))=\psi(x,\mR(T))$.
    Thus, $\chi(\mR(T))=\psi(\mR(T))$.
\end{theorem}
\begin{proof}
    Let $\kappa=\psi(x,\mR(T))$ and let $(U_\alpha:\alpha<\kappa)$ be a collection of open sets in $\mR(T)$ such that $\bigcap_{\alpha<\kappa} U_\alpha=\set{x}$.

    For each $\alpha<\kappa$, let $[t_\alpha, F_\alpha]$ be a basic open set such that $x\in [t_\alpha, F_\alpha]\subseteq U_\alpha$.

    Note that if $u$ is a top of $x$ which is not in $\bigcup_{\alpha<\kappa} F_\alpha$, then no ray of $T$ can contain $u$, for if $y$ is such ray, then $y \in \bigcap_{\alpha<\kappa} [t_\alpha, F_\alpha] =\{x\}$, but $y\neq x$ as $u\in y$ and $u\notin x$.

    \textbf{Case 1}: the set $\{t_\alpha:\alpha<\kappa\}$ is unbounded in $x$.

    In this case, $\mathcal B=\{[t_\alpha, F]: \alpha<\kappa, F=\bigcup_{\beta \in I} F_\beta \text{ for some finite } I\subseteq \kappa\}$ has cardinality $\leq \kappa$ and is a local base for $x$ in $\mR(T)$.
    To see that, let $[s, K]$ be a basic open set such that $x\in [s, K]$.

    Let $K'=K\cap \bigcup_{\alpha<\kappa} F_\alpha$, which is finite, so fix $I\subseteq \kappa$ finite such that $K' \subseteq \bigcup_{\beta \in I} F_\beta$.
    Since $\{t_\alpha:\alpha<\kappa\}$ is unbounded in $x$, there exists some $\alpha<\kappa$ such that $s\leq t_\alpha$.

    Then $[t_\alpha, \bigcup_{\beta \in I} F_\beta] \in \mathcal B$.
    We show that $[t_\alpha, \bigcup_{\beta \in I} F_\beta] \subseteq [s,K]$.
    Indeed, given $y\in [t_\alpha, \bigcup_{\beta \in I} F_\beta]$, as $s\leq t_\alpha\in y$, we have $s\in y$.
    Moreover, if $y\cap K\neq \emptyset$, then either $y\cap K'\neq \emptyset$, so $y\cap \bigcup_{\beta \in I} F_\beta \neq \emptyset$, which gives us a contradiction, or $y\cap (K\setminus K')\neq \emptyset$, so $y\cap \bigcup_{\alpha<\kappa} F_\alpha \neq \emptyset$, which is also a contradiction by our previous observation.

    \textbf{Case 2}: the set $\{t_\alpha:\alpha<\kappa\}$ is bounded in $x$.

    In this case, let $\bar t \in x$ be such that $t_\alpha\leq \bar t$ for every $\alpha<\kappa$.
    Then the family $\mathcal B=\{[\bar t, F]: F=\bigcup_{\beta \in I} F_\beta \text{ for some finite } I\subseteq \kappa\}$ has cardinality $\leq \kappa$ and is a local base for $x$ in $\mR(T)$.

    To see that, let $[s, K]$ be a basic open set such that $x\in [s, K]$ and define $F=\bigcup_{\beta \in I} F_\beta$ exactly as in Case 1.

    Let $y \in [\bar t, F]$.
    We show that $y\in [s,K]$.
    As in Case 1, $y\cap K\neq \emptyset$ would lead to $y \cap F\neq \emptyset$, a contradiction.

    If $s\leq \bar t$, then $s \in y$ and we are done.
    Otherwise, $\bar t<s$.
    If $s \notin y$, then $y\cap \bigcup_{\alpha<\kappa} F_\alpha=\emptyset$ (or we would have $x\subseteq y$ and therefore $s \in y$).
    Thus, $y\in \bigcap_{\alpha<\kappa} [t_\alpha, F_\alpha]=\{x\}$, again, a contradiction.
\end{proof}

One of the equivalences shown in Theorem 3.2 of \cite{kurkofka2022countablydeterminedends} is that, for all $\varepsilon \in \Omega(G)$, $\chi(\varepsilon, \Omega(G))$ is countable if, and only if, $\psi(\varepsilon, \Omega(G))$ is countable.
We thus highlight that the following corollary of \myref{THM_CharPseudoCharRays} generalizes that equivalence to arbitrary cardinals:

\begin{corollary}\label{CharPseudoCharEnds}
    For every graph $G$ and $\varepsilon\in \Omega(G)$, we have that $\chi(\varepsilon,\Omega(G))=\psi(\varepsilon,\Omega(G))$.
    Thus, $\chi(\Omega(G))=\psi(\Omega(G))$.
\end{corollary}

Now let us turn our attention to countable ray spaces:

\begin{lemma}\label{lemma:countablerays}
    A pruned rooted tree $T$ has countably many rays if, and only if, all the following conditions hold:
    \begin{enumerate}[label=(\roman*)]
        \item\label{item_counttops}  every ray of $T$ has at most countably many tops,
        \item\label{item_countbranch}  every branch of $T$ is countable,
        \item\label{item_countsucc} every node of $T$ has at most countably many successors,
        \item\label{item_no binary} $T$ contains no subset which is order-isomorphic to the binary tree $2^{<\omega}$.    
    \end{enumerate}
\end{lemma}
\begin{proof}
    It is clear that if $T$ has countably many rays, then \ref{item_counttops}, \ref{item_countbranch}, \ref{item_countsucc} and \ref{item_no binary} hold.

    On the other hand, suppose that $T$ has uncountably many rays and that \ref{item_counttops}, \ref{item_countbranch} and \ref{item_countsucc} hold. 
    
    We will show that \ref{item_no binary} fails.
    To start off, we claim that there exist two incomparable nodes $s_0, s_1 \in T$ such that each of the rooted trees $\lfloor s_0\rfloor$ and $\lfloor s_1\rfloor$ has uncountably many rays.
    Indeed, if this were not the case, then the set 
    \begin{equation*}
        \set{s\in T:  \mR(\lfloor s\rfloor) \text{ is uncountable}}= R
    \end{equation*}
    would be a ray of $T$.
    Furthermore, since every ray in $T$ has countably many tops, it follows that $R$ would actually have to be a branch of $T$.
    However, by \ref{item_countbranch}, $R$ is countable, so it follows from \ref{item_countsucc} that there must be some $t\in T\setminus R$ such that $\mR(\lfloor t\rfloor)$ is uncountable, a contradiction.

    It is clear that we may recursively proceed this way to construct an order embedding $\phi:2^{<\omega}\rightarrow T$, as desired.
\end{proof}

We note that the hypothesis that $T$ is pruned can be removed, as long as in (i) and (iii), we consider only the tops and successors which lie in some ray of $T$.
With that in mind, we have the following topological characterization of countable ray spaces:

\begin{corollary}\label{corollary:CountableRayRothberger}
    Let $T$ be a rooted tree.
    The following are equivalent:
    \begin{enumerate}[label=(\alph*)]
        \item\label{item_countableRays} $T$ has countably many rays.
        \item\label{item_RaysGdelta} $\mR(T)$ is Rothberger and $\psi(\mR(T))=\aleph_0$ (that is, every ray in $T$ is the intersection of countably many open sets).
        \item \label{item_RaysRoth} $\mR(T)$ is Rothberger and $\chi(\mR(T))= \aleph_0$ (that is, every ray in $T$ has a countable local base).
    \end{enumerate}
\end{corollary}
\begin{proof}
    The equivalence between \ref{item_RaysGdelta} and \ref{item_RaysRoth} follows directly from \myref{THM_CharPseudoCharRays}.

    The equivalence between \ref{item_countableRays} and \ref{item_RaysGdelta} follows from \myref{lemma:countablerays} and the trivial observation that conditions \ref{item_counttops} and \ref{item_countbranch} in \myref{lemma:countablerays} are, together, equivalent to every point in $\mR(T)$ being the intersection of countably many open sets, while conditions \ref{item_countsucc} and \ref{item_no binary} in \myref{lemma:countablerays} are, in view of \myref{THM_RaySpaceRothberger}, equivalent to  $\mR(T)$ being Rothberger.
\end{proof}

At last:

\begin{corollary}\label{COR_CountableEnds_SeqFin}
    A graph $G$ has countably many ends if, and only if, all the following conditions hold:
    \begin{enumerate}[label=(\roman*)]
        \item\label{item_EndRoth} For every $\subseteq$-increasing sequence $(F_n:n\in\omega)$ of finite subsets of $\mathrm{V}(G)$ there is a sequence $(C_n:n\in\omega)$, where each $C_n$ is a connected component of $G\setminus F_n$, so that every ray of $G$ has a tail in $C_n$ for some $n\in\omega$.
        \item\label{item_EndGdelta} For every ray $R$ in $G$ there exists a $\subseteq$-increasing sequence $(F_n:n\in\omega)$ of finite subsets of $\mathrm{V}(G)$ such that, for every ray $R'$ in $G$ which is not equivalent to $R$, there exists an $n\in\omega$ such that $R$ and $R'$ lie in different connected components of $G\setminus F_n$.
    \end{enumerate}
\end{corollary}
\begin{proof}
    It is clear that \ref{item_EndGdelta} is equivalent to every end of $G$ being the intersection of countably many basic open sets in $\Omega(G)$.
    Thus, the proof follows directly from \myref{COR_RothbergerEnds} and \myref{corollary:CountableRayRothberger}, in view of \myref{theorem:representation}.
\end{proof}

\begin{corollary}\label{COR_CountableEnds_BinTree}
    A graph $G$ has countably many ends if, and only if, all the following conditions hold:
    \begin{enumerate}[label=(\roman*)]
        \item\label{item_EndSubDiv} $G$ contains no end-faithful subgraph which is a subdivision of the binary tree.
        \item\label{item_EndGdelta1} For every ray $R$ in $G$ there exists a $\subseteq$-increasing sequence $(F_n:n\in\omega)$ of finite subsets of $\mathrm{V}(G)$ such that, for every ray $R'$ in $G$ which is not equivalent to $R$, there exists an $n\in\omega$ such that $R$ and $R'$ lie in different connected components of $G\setminus F_n$.
    \end{enumerate}
\end{corollary}

As a footnote of this section, we note that condition \ref{item_EndGdelta} above could, by a diagonalization argument, be changed by the apparently strongly condition:

\begin{enumerate}[label=(\roman*')]
    \setcounter{enumi}{1}
    \item\label{item_EndCountBase} There exists a $\subseteq$-increasing sequence $(F_n:n\in\omega)$ of finite subsets of $\mathrm{V}(G)$ such that for every two distinct rays $R$ and $R'$ in $G$, there exists an $n\in\omega$ such that $R$ and $R'$ lie in different connected components of $G\setminus F_n$.
\end{enumerate}
\section{On the Menger property}
In this section, we adapt the techniques used in the previous section to study the Menger property for end spaces and ray spaces.
In particular, we show that for these spaces, the Menger property is equivalent to $\sigma$-compactness --  showing that Menger's conjecture holds for these classes of spaces -- and that these properties are equivalent to the spaces not having a closed copy of the irrationals.

To prove the results in this section we will need the following auxiliary definitions and results.

\begin{definition}\label{DEF_K-BaireRank}
    Suppose $T$ is a rooted tree and $S\subseteq T$ is a subtree. We will say that $t\in T$ is \emph{$S$-compactly trivial} if every $s\in S$ above $t$ has finitely many successors in $S$.

    Moreover, we recursively define, for each ordinal $\alpha$, a subtree $\partial_K^\alpha(T)$ of $T$ as follows:
    \begin{itemize}
        \item $\partial_K^0(T)=T$;
        \item If $\partial_K^\alpha(T)$ has been defined, let $\partial_K^{\alpha+1}(T)$ be the subtree of $\partial_K^\alpha(T)$ obtained by removing all $\partial_K^\alpha(T)$-compactly trivial nodes from $\partial_K^\alpha(T)$;
        \item If $\beta$ is a limit ordinal and $\partial_K^\alpha(T)$ has been defined for all $\alpha<\beta$, let $\partial_K^\beta(T) = \bigcap_{\alpha<\beta} \partial_K^\alpha(T)$.
    \end{itemize}
    If there exists $\gamma$ such that $\partial_K^\gamma(T)=\emptyset$, then we call the least such $\gamma$ the \emph{$K$-Baire rank} of $T$ and denote the smallest such $\gamma$ by $\rk_K(T)$.
    In this case, for every $t \in T$, the supremum of all $\alpha$'s for which $t \in \partial_K^\alpha(T)$ is called the \emph{$K$-Baire rank} of $t$ and denoted by $\rk_K(T, t)$.
    Notice that the supremum is attained, so it is, in fact, a maximum.
\end{definition}

For the next result, notice that if $S\subseteq T$ is a subtree of $T$, then $\mathcal R(S)\subseteq \mathcal R(T)$ and the ray space topology of $\mathcal R(S)$ agrees with the subspace topology inherited from $\mathcal R(T)$.
\begin{theorem}\label{THM_RayClosedBaire}
    For a rooted tree $T$, the following are equivalent:
    \begin{enumerate}[label=(\alph*)]
        \item $\mR(T)$ does not contain a closed topological copy of the Baire space $\omega^\omega$;
        \item There is no order embedding $\varphi\colon \omega^{<\omega}\to T$ such that for every $s\in \omega^{<\omega}$ there is a $t_s\in T$ with $t_s\geq\varphi(s)$ for which, for every $n\in\omega$, $\varphi(s^\smallfrown n)$ is a successor of $t_s$.
        \item There exists an ordinal $\gamma$ such that $\partial_K^\gamma(T)= \emptyset$.
        \item For every nonempty closed $X\subseteq \mR(T)$ there exists some point $x\in X$ and a set $V\subseteq X$ with $x\in V$ which is open and compact in $X$.
    \end{enumerate}
\end{theorem}
\begin{proof}
    The fact that (d) implies (a) holds as a closed copy of $\omega^\omega$ would contradict (a), and the fact that (b) implies (c) holds as the range of $\varphi$ would be contained in $T_\gamma$ for every ordinal $\gamma$.

    Now we prove that (a) implies (b) by its contrapositive.
    Suppose that $T$ contains an order-isomorphic copy $S$ of $\omega^{<\omega}$ as in (b), with order isomorphism $s\in \omega^{<\omega} \mapsto t_s\in S$.
    Then it is clear that the map $f\in \omega^\omega \mapsto R_f\in \mR(T)$, where $R_f$ is the (unique) ray in which $\set{t_{f\restrict n}:n\in\omega}$ is cofinal, is a topological embedding.
    Furthermore, we claim that the union 
    \begin{equation*}
    \bigcup_{s\in \omega^{n}}[t_s]
    \end{equation*}
    is closed in $\mR(T)$ for every $n\in\omega$.
    Indeed, this is clear for $n=0$, as $[t_{\emptyset}]$ is clopen in $\mR(T)$.
    Now suppose that the claim holds for $n\ge 1$ and let $R\in \overline{\bigcup_{s\in \omega^{n+1}}[t_s]}$.
    Since $$\bigcup_{s\in \omega^{n+1}}[t_s] \subseteq \bigcup_{r\in \omega^{n}}[t_r],$$
    it follows from our induction hypothesis that there is an $r\in \omega^n$ such that $R\in [t_r]$.
    On the other hand, by our assumption over $S$, it is clear that $\bigcup_{k\in\omega}[t_{r^\smallfrown k}]$ is closed in $\mR(T)$, so it follows that $R\in [t_{r^\smallfrown m}]$ for some $m\in\omega$, concluding the proof of our claim.

    At last, in order to show that the image of the defined map is closed, suppose that $R\in \overline{\set{R_f:f\in \omega^\omega}}$.
    We recursively define $g\in \omega^\omega$ such that $R\in[t_{g\restrict n}]$ for every $n\in\omega$ as follows.
    Firstly, it is clear that $R\in [t_{\seq{\,}}]$.
    Now suppose that we have found an $s=g|n\in \omega^{n}$ for which $R\in [t_s]$.
    Since 
    $$\bigcup_{r\in \omega^{n+1}}[t_r]$$ 
    is a closed set in $\mR(T)$ containing $\set{R_f:f\in \omega^\omega}$, there must be a (unique) $m\in \omega$ such that $R\in [t_{s^\smallfrown m}]$.
    Thus, letting $g(n)=m$ concludes the recursive construction of $g$.

    Clearly, $R_g\subseteq R$. If $R_g\subseteq R$, let $x \in R$ be a top of $R_g$.
    Then $[x]$ is a standard basic open neighborhood of $R$ which is disjoint from $\set{R_f:f\in \omega^\omega}$, contradicting our assumption on $R$.
    Thus, $R=R_g\in \{R_f: f \in \omega^\omega\}$.
    This finishes the proof that (a) implies (b).

    Now we prove that (b) implies (c) by its contrapositive.
    Assume $t \in \partial_K^\gamma(T)$.
    Then, by the definition of $\partial_K^{\gamma+1}(T)$, it follows that $t$ is not $\partial_K^\gamma(T)$-compactly trivial, so there exists $s \in \partial_K^\gamma(T)$ above $t$ with infinitely many successors in $\partial_K^\gamma(T)$.
    Let $\phi(\emptyset)=t$ and $t_\emptyset=s$.

    Continuing in this manner, we thus recursively construct the desired $\varphi\colon \omega^{<\omega}\to T$.

    Finally, we prove that (c) implies (d).
    Suppose that $X\subseteq \mR(T)$ is nonempty and closed. 
    Let $\alpha\leq\gamma$ be the least ordinal such that $\bigcup X\not\subseteq \partial_K^\alpha(T)$.
    Notice that $\alpha\neq 0$ and $\alpha$ is not a limit, so there exists $\beta<\gamma$ such that $\alpha=\beta+1$.
    Let $t \in \bigcup X\setminus \mR(\partial_K^\alpha(T))$ and fix $x \in X$ such that $t\in x$.
    We claim that the clopen neighborhood $[t]\cap X$ of $x \in X$ is compact.
    
    Let $S=\{s \in \partial_K^\alpha(T): s\leq t\, \vee\, t\leq s\}$. Then $S$ is a subtree of $T$, therefore $\mathcal R(S)\subseteq\mathcal R(T)$ and the ray space topology of $\mathcal R(S)$ agrees with the subspace topology inherited from $\mathcal R(T)$.
    Since $t$ is $\partial_K^\alpha(T)$-compactly trivial in $\partial_K^\alpha(T)$, it follows from Lemma~\ref{compactFiniteSuccessors} that $\mathcal R(S)$ is a compact subspace of $\mathcal R(T)$.
    As $\bigcup X\subseteq \partial_K^\alpha(T)$, we have $X\cap [t]\subseteq \mathcal R(S)$.
    Moreover, both $X$ and $[t]$ are closed in $\mathcal R(T)$, so $X\cap [t]$ is closed in $\mathcal R(S)$, which is compact.
    Thus, $X\cap [t]$ is compact.
\end{proof}

\begin{lemma}\label{lemma:RaylessPartition}
    Let $T$ be a rooted tree and suppose $([t_i,F_i]:i\in I)$ is a partition of $\mathcal{R}(T)$ comprising standard basic open sets of $\mathcal{R}(T)$.
    Then the tree $\{t_i:i\in I\}$ with the induced suborder is rayless.
\end{lemma}
\begin{proof}
    Striving for a contradiction, suppose $R\subseteq\{t_i:i\in I\}$ is a minimal ray. 
    Then $R$ is cofinal in some ray $R'\in\mathcal{R}(T)$.
    Since $\{[t_i,F_i]:i\in I\}$ is a partition of $\mathcal{R}(T)$, there must be a $\delta\in I$ such that $R'\in [t_\delta,F_\delta]$. 
    On the other hand, as $R$ is cofinal in $R'$, there must be a $\gamma\in I$ such that $t_\gamma\in R$ and $t_\gamma>t_\delta$.
    But, since $[t_\delta,F_\delta]\cap [t_\gamma,F_\gamma]=\emptyset$, this means that $t_\gamma$ is above some $s\in F_\delta$, contradicting the fact that $R'\in [t_\delta,F_\delta]$.
\end{proof}

The following lemma is similar to Lemma~\ref{lemma:split}.
\begin{lemma}\label{lemma:notMenger}
    Let $X$ be a topological space and $(\mathcal{U}_n: n\in A)$, where $A$ is countable, be a sequence of countable open covers of $X$ which shows that $X$ is not Menger -- that is, for every sequence $(\mathcal{F}_n: n\in A)$ such that $\mathcal{F}_n\subset \mathcal{U}_n$ is finite for all $n\in A$, there exists $x\in X$ such that $x\notin \bigcup_{n\in A} \bigcup\mathcal{F}_n$.

    Then for every $N\in A$ there exists an infinite $\mathcal{V}\subset \mathcal{U}_N$ such that for each $V\in \mathcal{V}$ there is an infinite  $B_V\subset A$ such that ($\mathcal{U}_n: n\in B$) attests that $V$ is not Menger.
\end{lemma}

\begin{proof}
    Assume the contrary, that is, there is an $N\in A$ such the set of all $V \in \mathcal U_N$ for which there are infinite $B_V\subset A$ such that ($\mathcal{U}_n: n\in B_V$) attests that $V$ is not Menger is finite, say $F$
    
    Enumerate $\mathcal U_N\setminus F$ as $(V_k:k\in\omega)$.
    Consider a partition $(B_k:k\in\omega)$ of $A\setminus \{N\}$ into infinite sets.
    For each $k\in\omega$, since $V_k$ is not in the aforementioned finite set, there is a sequence $(\mathcal{F}_n:n\in B_k)$ such that each $\mathcal{F}_n\subset \mathcal{U}_n$ is finite and $V_k\subseteq \bigcup_{n\in B_k} \bigcup \mathcal{F}_n$.

    Then $X=\bigcup_{V\in F} V \cup \bigcup_{k\in\omega} \bigcup_{n\in B_k} \bigcup \mathcal{F}_n$, showing that $(\mathcal{U}_n: n\in A)$ does not witness that $X$ is not Menger, a contradiction.
\end{proof}

If $T$ is a rooted tree and $s \in T$, let $T_{(s)}=\{a \in T: s\leq a\, \vee \, a\leq s\}$ be the subtree of $T$ consisting of all nodes comparable with $s$.

\begin{lemma}\label{LEMMA_Compact-Sigma_glueing}
    Suppose $T$ is a rooted tree and that $S$ is a subtree of $R$.
    For each $s\in S$, let $F_s$ be a countable collection of direct successors of $s$ in $T\setminus S$
    For each ray $x\subseteq S$, let $F_x$ be a collection (of any cardinality) of tops of $x$ in $T\setminus S$.

    Assume that:
    \begin{enumerate}[label=(\alph*)]
        \item For every $s\in S$ and $u \in F_s$, $S_u\subseteq T_{(u)}$ is a subtree of $T$ such that $\mR(S_u)$ is $\sigma$-compact.
        \item For every $x \in \mathcal R(S)$ and $u \in F_x$, $S_u\subseteq T_{(u)}$ is a subtree of $T$ such that $\mR(S_u)$ is $\sigma$-compact.
        \item $\mathcal R(S)$ is compact.
   
    \end{enumerate}
    Then $\mathcal R(S\cup\bigcup_{s \in S}\bigcup_{u \in F_s} S_u\cup \bigcup_{x \in \mathcal R(S)}\bigcup_{u \in F_x} S_u)$ is $\sigma$-compact.
\end{lemma}
\begin{proof}
    For each $x \in S \cup \mathcal R(S)$ and $u \in F_x$, as $\mathcal R(S_u)$ is $\sigma$-compact, let $(K_u^n: n\in\omega)$ be a sequence of compact subsets of $\mathcal R(S_u)$ such that $\mathcal R(S_u)=\bigcup_{n\in\omega} K_u^n$.
    By Lemma~\ref{LEMMA_Compact-DownClosure}, we may assume that each $K_u^n$ is of the form $\mathcal R(S_{u, n})$, where $S_{u,n}\subseteq S_u$ is a subtree of $S_u$.
    Thus, $\mathcal R(S_u)=\bigcup_{n\in\omega} \mathcal R(S_{u, n})$, where $S_{u,n}\subseteq S_u$ is a subtree of $S_u$ such that $\mathcal R(S_{u,n})$ is compact.

    Let $S'=\{s \in S: F_s\neq \emptyset\}$.
    For each $s \in S'$, let $F_s=\{u_{s,k}: k\in\omega\}$ be an enumeration of $F_s$ (repeating if necessary).

    For each $k, n \in \omega$, let $T_{k, n}=S\cup \bigcup_{k, n}S_{u_{s,k}, n}$.
    Each $T_{k, n}$ is compact: given $s \in S$, $s$ has finitely many successors in $S$, and at most one successor not in $S$ (that is, $u_{s, k}$).
    Moreover, given $s\in T_{k, n}\setminus s$, $s$ in exactly one of the $S_{u_{s, k}, n}$'s, so it has finitely many successors in $S_{u_{s, k}, n}$, and therefore in $T_{k, n}$.
    By Lemma~\ref{compactFiniteSuccessors}, $\mathcal R(T_{k, n})$ is compact.

    Also, for each $n \in \omega$, let  $T'_{n}=S\cup \bigcup_{x \in \mathcal R(S)}\bigcup_{u \in F_x} S_{u, n}$.
    We verify that $\mathcal R(T'_{n})$ is compact, again by Lemma~\ref{compactFiniteSuccessors}.
    Given $s \in S$, $s$ has finitely many successors in $S$ and no successor in $T'_{n}\setminus S$.
    Given $t\in T'_{n}\setminus S$, $t$ has no successor in $S$ and is in exactly one of the $S_{u, n}$'s, so it has finitely many successors in $S_{u, n}$, and therefore in $T'_{n}$.
    Thus, $\mathcal R(T'_{n})$ is compact.

    Now it suffices to see that $\mathcal R(S\cup\bigcup_{s \in S}\bigcup_{u \in F_s} S_u\cup \bigcup_{x \in \mathcal R(S)}\bigcup_{u \in F_x} S_u)=\bigcup_{k, n \in \omega} \mathcal R(T_{k, n})\cup \bigcup_{n \in \omega} \mathcal R(T'_{n})\cup \mathcal R(S)$.

    The inclusion from right to left is trivial.
    For the other inclusion, notice that if $x$ is in the left and $x\notin \mathcal R(S)$, then $x$ intersects some $F_x$ for $x\in S\cup \mathcal R(S)$ at some point $u$.
    Then $x \in \mathcal R(S_u)$, so $x\in \mathcal R(S_{u, n})$ for some $n\in\omega$.
    If $x \in S$, then $x=u_{x, k}$ for some $k\in\omega$, so $x\in \mathcal R(T_{k, n})$.
    If $x \in \mathcal R(S)$, then $x\in \mathcal R(T'_{n})$.
\end{proof}

Now we are ready to present the main result of this section.
To motivate it, we recall the classical result in descriptive set theory due to Hurewicz \cite[Theorem 7.10]{kechris2012classical} stating that a Polish space is $\sigma$-compact if, and only if it does not contain a closed topological copy of the irrationals, $\omega^\omega$.
\begin{theorem}\label{THM_RayMengerCharacterization}
    For every rooted tree \( T \), the following are equivalent:
    \begin{enumerate}[label=(\alph*)]
        \item $\mR(T)$ is $\sigma$-compact.  \label{theorem:RayMengerCharacterization_locallyPiCompact}
        \item $\mathcal{R}(T)$ is Menger.\label{theorem:RayMengerCharacterization_menger}
        \item $\mR(T)$ is Lindel\"of and there is no order embedding $\varphi\colon \omega^{<\omega}\to T$ such that for every $s\in \omega^{<\omega}$ there is a $t_s\in T$ with $t_s\geq\varphi(s)$ for which, for every $n\in\omega$, $\varphi(s^\smallfrown n)$ is a successor of $t_s$.\label{theorem:RayMengerCharacterization_combinatorial}
        \item $\mathcal{R}(T)$ is Lindel\"of and does not contain a closed topological copy of the Baire space $\omega^\omega$. \label{theorem:RayMengerCharacterization_copy}
        \item $\mathcal{R}(T)$ is Lindel\"of and, for every nonempty closed $X\subseteq \mR(T)$, there is an $x\in X$ with an open neighborhood $V_x$ which is compact in $X$.
        \item $\mathcal{R}(T)$ is Lindel\"of and there exists an ordinal $\gamma$ such that $\partial_K^\gamma(T)= \emptyset$.
    \end{enumerate}

    Moreover, if $T$ is a pruned tree, then the ``Lindel\"of'' condition in all items may be swapped by ``every node has at most countably many successors''.
\end{theorem}

Remark: Items (a) and (b) say that in the realm of ray spaces, Menger spaces are exactly the $\sigma$-compact spaces.
Item (c) may be seen as a combinatorial characterization of Menger ray spaces, while items (d), (e) and (f) are topological characterizations.
\begin{proof}
    Property (a) implies (b) in every topological space, and the fact that (b) implies (d) is clear, as $\omega^\omega$ is not Menger and closed subspaces of Menger spaces are Menger. 
    Moreover, the equivalences between (c), (d), (e) and (f) are direct corollaries of \myref{THM_RayClosedBaire}.
    The final sentence of the theorem follows from \cite[Proposition~2.16]{kurkofka2024representationtheoremendspaces} (as stated in Proposition~\ref{proposition:extentTrees}).
    Thus, it suffices to prove that (f) implies (a).

    For each $t \in T$, let $T_{(t)}=\{u \in T: u\leq t \, \vee t \leq u\}$.
    We recursively prove that for every $\alpha\leq \rk_K(T)$ and every $t \in T$ with $\rk_K(T,t)=\alpha$, the space $\mR(T_{(t)})$ is $\sigma$-compact.
    Then we get our conclusion by letting $t$ be the root of $T$.

    Assume that $\rk_K(T,t)=\alpha$ and that the claim holds for every $\beta<\alpha$.
    Let $S=T_{(t)}\cap \partial_{K}^\alpha(T)$.
    As $t$ is compactly trivial in $\partial_K^\alpha(T)$, it follows from Lemma~\ref{compactFiniteSuccessors} that $\mR(S)$ is compact.

    For each $s \in S$, let $F_s$ be the collection of all direct successors of $s$ in $T_{(t)}\setminus S$ which belong to some ray in $T_{(t)}$.
    For each $x \in \mathcal R(S)$, let $F_x$ be the collection of all tops of $x$ in $T_{(t)}\setminus S$.

    As $\mathcal R(T)$ is Lindel\"of and $\mathcal R(T_{(t)})$ is a closed subspace of $\mathcal R(T)$, the latter is also Lindelöf.
    Given $s \in S$, the set $C=\{x \in \mathcal R(T_t): s \in x\}$ is closed in $\mathcal R(t_T)$, thus, is also Lindelöf.
    It may be partitioned as $C=\bigcup_{u \in F_s} \{x \in \mathcal R(T_{(t)}): u \in x\}$, thus, by Lindelöfness, $F_s$ is countable.

    As every element $u$ of each $F_x$ and $F_s$ is not in $\partial_K^\alpha(T)$, it follows from the induction hypothesis that for each such $u$, $\mR(T_{(u)})$ is $\sigma$-compact.
    Thus, by Lemma~\ref{LEMMA_Compact-Sigma_glueing}, the following set is $\sigma$-compact:

    \begin{equation*}
    \mR\left(S\cup\bigcup_{s \in S}\bigcup_{u \in F_s} T_{(u)}\cup \bigcup_{x \in \mathcal R(S)}\bigcup_{u \in F_x} T_{(u)}\right).
    \end{equation*}

    As this ray space is a subspace of $\mR(T_{(t)})$, we only need to verify that $\mR(T_{(t)})$ is contained in it.

    Let $y \in \mR(T_{(t)})$.
    If $y\subseteq S$, then $y \in \mR(S)$ and we are done.
    If not, let $u$ be the first point of $y$ not in $S$.
    Then $u$ is either a top of some ray in $S$ or a direct successor of some $s \in S$ not in $S$ belonging to some ray (say, $y$ itself).
    In any case, it is clear that $y \in \mR(T_{(u)})$ and we are done.
\end{proof}
Thus, the representation theorem for end spaces gives us the following corollary.
\begin{corollary}\label{COR_EndMengerCharacterization}
    If \( G \) is a graph, then the following are equivalent:
    
    \begin{enumerate}[label=(\alph*)]
        \item $\Omega(G)$ is $\sigma$-compact.
        \item $\Omega(G)$ is Menger.
        \item $\Omega(G)$ is Lindelöf and does not contain a closed topological copy of the Baire space $\omega^\omega$.
    \end{enumerate}
\end{corollary}
\begin{proof}
    This follows from \myref{THM_RayMengerCharacterization} and the fact that every end space can be represented as the ray space of a rooted tree (as shown in \myref{theorem:representation}).
\end{proof}

We should highlight the importance of the ``closed'' assumption in Theorem \ref{THM_RayMengerCharacterization}(d) and Corollary~\ref{COR_EndMengerCharacterization}(c), as the following examples illustrate:

\begin{example}
    The tree $2^{<\omega}$ is atomless, countable and nonempty, therefore contains a dense subset order-isomorphic to $\omega^{<\omega}$ (whose set of branches is thus a topological copy of $\omega^\omega$).
    However, the ray space of $2^{<\omega}$ is homeomorphic to the Cantor space, which is compact (and therefore Menger). 
\end{example}

\begin{example}\label{EX_MengerWithBaireSpace}
    Consider $G$ as the graph such that $\mathrm{V}(G) = \omega^{<\omega}$ and two vertices $s,t\in \mathrm{V}(G)$ are adjacent if, and only if, one of the following conditions apply:
    \begin{itemize}
        \item $s$ is a successor of $t$ or vice versa;
        \item $s$ and $t$ are both successors of some $r\in \omega^{<\omega}$ and $s\neq t$.
    \end{itemize}
    Then it is clear that $\Omega(G)$ is compact (and therefore Menger), but it nevertheless contains an obvious topological copy of the Baire space $\omega^\omega$.
\end{example}

Furthermore, one could be tempted to conjecture that the following combinatorial statement about a graph $G$ is, analogously to the Rothberger property, equivalent to $\Omega(G)$ being Menger:
\begin{itemize}
    \item[(F)] For every sequence \(\seq{F_n:n\in\omega}\) of increasing finite subsets of $\mathrm{V}(G)$ there is a sequence \(\seq{\mathcal{F}_n:n\in\omega}\) such that each $\mathcal{F}_n$ is a finite collection of connected components of $G\setminus F_n$ and every $G$-ray has a tail in $C\in \mathcal{F}_n$ for some $n\in\omega$.
\end{itemize}
However, the following result shows us that this is not the case.

\begin{theorem}\label{theorem:lindelofGraph2}
    Let $G$ be a graph. Then $\Omega(G)$ is Lindel\"of if, and only if, $G$ satisfies (F).
\end{theorem}
\begin{proof}
    It is clear that if $\Omega(G)$ satisfies (F), then it is Lindel\"of.

    Now suppose that $\Omega(G)$ is Lindel\"of and let \(\seq{F_n:n\in\omega}\) be a sequence of increasing finite subsets of $\mathrm{V}(G)$.
    We may assume it is strictly increasing (otherwise it has a constant subsequence and the proof is easy).
    We construct \(\seq{\mathcal{F}_n:n\in\omega}\) as follows: 
    firstly, in view of \myref{theorem:lindelof}, we may fix an enumeration $\set{C_n:n\in\omega}$ for the connected components of $G\setminus F_0$ which contain rays.
    Then, for each $n\in\omega$, let $k_n\in\omega$ be the largest integer such that $C_{k_n}\cap \left(F_{n+1}\setminus F_0\right)\neq \emptyset$.
    Note that, in this case, $C_n$ is a connected component of $G\setminus F_m$ for every $n> k_m$.
    Thus, the sequence $\seq{\mF_n:n\in\omega}$ with
    \begin{align*}
        \mF_0 &= \set{C_0, \dotsc, C_{k_0}},\\
        \mF_{n+1} &= \set{C_{k_{n}+1}, \dotsc, C_{k_{n+1}}},
    \end{align*} 
    is as desired.
\end{proof}

We end this section with a final remark.
So far, we have focused our analysis in the Menger property.
Looking at $\sigma$-compactness instead, it is possible to adapt Lemma~\ref{compactFiniteSuccessors} to derive a combinatorial characterization (relying once again on \myref{lemma:compactComponents}):

\begin{proposition}\label{proposition:sigmaCompactEnds}
    Let $G$ be a graph.
    Then $\Omega(G)$ is $\sigma$-compact if, and only if, there exists a countable family $\{H_n:n\in\omega\}$ of induced subgraphs of $G$ such that:
    \begin{enumerate}[label=(\alph*)]
        \item For every $n\in\omega$ and finite $F\subseteq G$ the number of connected components of $H_n\setminus F$ which contain rays is finite.
        \item $H_n$ is end-faithful (i.e., the inclusion mapping $i:\Omega(H_n)\to \Omega(G)$ is well-defined and injective).
        \item Every ray of $G$ is equivalent to a ray of $H_n$ for some $n\in\omega$.
    \end{enumerate}
\end{proposition}

\section{On edge-end spaces}\label{sec:edges}

In this section, we adapt our results for end spaces to the context of edge-end spaces.

Edge-end spaces arise as a natural generalization of classical end spaces, particularly when considering graphs that are not locally finite.
While end spaces are defined via vertex separations, edge-end spaces instead use separations by finite sets of edges, leading to a coarser equivalence relation among rays.
This perspective is especially relevant in infinite graph theory, where edge-based connectivity often reveals subtler topological properties and distinctions not captured by vertex-based definitions.

The following definitions and results set the stage for this analysis.
\begin{definition}[Edge-end space]
    Let $G$ be a graph.
    \begin{enumerate}[label=(\alph*)]
        \item We say that two rays $R$ and $R'$ in $G$ are \emph{edge-equivalent} if for every finite set $F\subseteq \rmE(G)$, there are tails $T$ and $T'$ of $R$ and $R'$, respectively, such that both $T$ and $T'$ lie on the same connected component of $G\setminus F$.
        \item The \emph{edge-end space} of $G$, denoted by $\Omega_E(G)$, is the set of all equivalence classes of rays in $G$.
        An element $\epsilon$ of $\Omega_E(G)$ is called an \emph{edge-end} of $G$.
        \item Let $F\subset \rmE(G)$ be a finite set, $C$ be a connected component of $G\setminus F$, and $\varepsilon\in \Omega_E(G)$.
        We define $\Omega_E(C, F)$, $\rmC(G, F)$, $G(\epsilon, F)$ and $\Omega_E(\epsilon, F)=\Omega_E(G(\epsilon, F), F)$ as in \myref{def:endspace}.
        \item We say that $U\subseteq \Omega_E(G)$ is \emph{open} if for every edge-end $\epsilon$ in $U$, there exists a finite $F\subset\rmE(G)$ such that $\Omega_E(\epsilon, F)\subseteq U$.    
        \item Furthermore, we say a vertex $v$ \emph{edge-dominates a ray} $r$ if no finite set of edges separate them, and it \emph{edge-dominates an edge-end} when it edge-dominates one of its representatives.
        \item Finally, we say that a vertex is \emph{timid} if it does not edge-dominate any ray in $G$. 
        We denote by $\mathrm{t}(G)$ the set of timid vertices of $G$.
    \end{enumerate}
\end{definition}

It should be clear that for locally finite graphs the end space and the edge-end space coincide.
On the other hand, one can easily come up with examples of non-locally finite graphs in which the edge-end equivalence is coarser than the one in end spaces.
\begin{theorem}[Theorem 3.1.4 in \cite{Boska2025}]
\label{THM_compactnesscharacterization}
    The edge-end space of a graph $G$ is compact if, and only if, for every finite set of timid vertices $F$, the collection ${\mathrm{C}}(G,F)$ of non-rayless connected components is finite.
\end{theorem}

Given a graph $G$, consider the following graph $H_G$, which is obtained by expanding edge-dominating vertices to cliques. Formally, for each edge-dominating vertex $v\in \mathrm{V}(G)$ we will add a complete graph $K_v \doteq K_{\deg{v}}$ with $\deg(v)$ many vertices.
This construction appears in \cite{aurichi2024topologicalremarksendedgeend}.
Each $v$-neighbour $u$ corresponds to a vertex $u^v \in\mathrm{V}(K_v)$.
Denoting by $D$ the set of vertices of $G$ that edge-dominate some ray, the vertex set of $H_G$ is given by \[\mathrm{V}(H_G) = (\mathrm{V}(G)\setminus D) \cup\displaystyle \bigcup_{v\in D}\mathrm{V}(K_v),\] while its edge set is defined by the following conditions:
\begin{itemize}
    \item[(i)] Add the edge $\{u_1^v,u_2^v\}$ for each pair of neighbors $u_1,u_2$ of the edge-dominating $v$, as to make $K_v$ into a complete sub-graph.
    \item[(ii)] We maintain the $G$ edges $\{u,v\}$ for timid pairs.
    \item[(iii)] Add the edge $\{u,u^v\}$ if $u$ is a timid neighbor of the dominating $v$.
    \item[(iv)] Add the edge $\{v^u,u^v\}$ if $u,v \in D$ are neighboring edge-dominating vertices.
\end{itemize}

We will rely on throughout the remainder of this section in the following results:
\begin{theorem}[Theorem 2.1 in \cite{aurichi2024topologicalremarksendedgeend}]\label{THM_HGHomeo}
For every graph $G$, $\Omega(H_G)\cong \Omega_E(G)$.
\end{theorem}

Indeed, we immediately get the following corollaries.

\begin{corollary}\label{CharPseudoCharEdgeEnds}
    For every graph $G$ and $\varepsilon\in \Omega_E(G)$, we have that $\chi(\varepsilon,\Omega_E(G))=\psi(\varepsilon,\Omega_E(G))$.
    Thus, $\chi(\Omega_E(G))=\psi(\Omega_E(G))$.
\end{corollary}

\begin{corollary}\label{COR_RothEdgeEnds}
    Let $G$ be a graph. Then, the following are equivalent:
    \begin{enumerate}[label=(\alph*)]
        \item $\Omega_E(G)$ is Rothberger.

        \item $\Omega_E(G)$ is Lindelöf and does not contain a copy of the Cantor space.

        \item $\Omega_E(G)$ is Lindelöf and scattered.

        \item For every strictly $\subseteq$-increasing sequence $(F_n:n\in\omega)$ of finite subsets of $\mathrm{E}(G)$ there is a sequence $(C_n:n\in\omega)$, where each $C_n$ is a connected component of $G\setminus F_n$, so that every ray of $G$ has a tail in $C_n$ for some $n\in\omega$.
    \end{enumerate}
\end{corollary}
\begin{proof}
    Equivalences (a)$\Leftrightarrow$(b)$\Leftrightarrow$(c) follow from \myref{THM_HGHomeo} and \myref{COR_RothbergerEnds}.

    (a) implies (d) is clear, since $\rmC(G,F)$ gives us an open cover of $\Omega_E(G)$ for every finite $F\subseteq \rmE(G)$.
    
    At last, the proof of (d) implies (b), is completely analogous to the proof of \myref{COR_RothbergerEnds}(d) implies (b) (one only needs to switch the finite sets of vertices for finite sets of edges in it).
\end{proof}

\begin{corollary}\label{COR_CountableEdgeEnds}
    A graph $G$ has countably many edge-ends if, and only if, all the following conditions hold:
    \begin{enumerate}[label=(\roman*)]
        \item\label{item_EdgeEndRoth} For every $\subseteq$-increasing sequence $(F_n:n\in\omega)$ of finite subsets of $\mathrm{E}(G)$ there is a sequence $(C_n:n\in\omega)$, where each $C_n$ is a connected component of $G\setminus F_n$, so that every ray of $G$ has a tail in $C_n$ for some $n\in\omega$.
        \item\label{item_EdgeEndGdelta} For every ray $R$ in $G$ there exists a $\subseteq$-increasing sequence $(F_n:n\in\omega)$ of finite subsets of $\mathrm{E}(G)$ such that, for every ray $R'$ in $G$ which is not edge-equivalent to $R$, there exists an $n\in\omega$ such that $R$ and $R'$ lie in different connected components of $G\setminus F_n$.
    \end{enumerate}
\end{corollary}
\begin{proof}
    It is clear that \ref{item_EdgeEndGdelta} is equivalent to every edge-end of $G$ being the intersection of countably many basic open sets in $\Omega_E(G)$.
    Thus, the proof follows directly from \myref{COR_RothEdgeEnds} and \myref{corollary:CountableRayRothberger}, in view of \myref{theorem:representation} and \myref{THM_HGHomeo}.
\end{proof}

\begin{corollary}\label{COR_MengerEdgeEnds}
    Let $G$ be a graph. Then, the following are equivalent:
    \begin{enumerate}[label=(\alph*)]
        \item $\Omega_E(G)$ is $\sigma$-compact.
        \item $\Omega_E(G)$ is Menger.
        \item $\Omega_E(G)$ is Lindelöf and does not contain a closed copy of the Baire space.
    \end{enumerate}
    Moreover, $\Omega_E(G)$ is a $D$-space.
\end{corollary}

Now suppose that $F$ is a finite set of timid vertices in $G$. Given a connected component $C\in \rmC(G,F)$ with $v\in C$, let $C_H\in \mathrm{C}(H_G,F)$ be the connected component containing $v$ if $v$ is timid, or, otherwise, containing $v_e$ for some edge $e\in C$ adjacent to $v$. We will thus rely on the following technical results for the final part of this section:

\begin{proposition}[Proposition 3.1.2 in \cite{Boska2025}]\label{PROP_Dominant_FinComp}
    Let $G = (V,E)$ be a connected graph and $e\in E$ be an edge adjacent to a vertex $v$ which edge-dominates some ray in $G$. Then $|\mathrm{C}(H_G, \{v_e\})|\le 2$.
\end{proposition}

\begin{proposition}[Lemma 3.1.3 in \cite{Boska2025}]\label{LEMMA_Compon_Bijec}
    If $F\subset \rmt(G)$ (and, thus, $F\subset \rmV(H_G)$) is finite, then the association 
    \begin{align*}
        \varphi\colon \: \mathrm{C}(G,F)\,&\to \,\mathrm{C}(H_G,F)\\
        C&\mapsto C_H
    \end{align*}
    is a well-defined bijection.
\end{proposition}

With these tools, we can now prove the last of the main results of this section: a combinatorial characterization of the Lindelöf degree of edge-end spaces.

\begin{corollary}\label{COR_LindelofEdgeEnds}
    Let $G$ be a graph. 
    Then, $L(\Omega_E(G))\le\kappa$ if, and only if, for every finite $F\subset \rmt(G)$, $|\rmC(G,F)|\le \kappa$.
\end{corollary}
\begin{proof}
    First, assume $L(\Omega_E(G))\le\kappa$ and let $F$ be a given finite set of timid vertices.
    By \myref{THM_HGHomeo}, $L(\Omega(H_G))\le\kappa$ as well. 
    Thus, by \myref{theorem:lindelofDegree}, $|\rmC(H_G, F)|\le \kappa$. 
    In this case, it follows from \myref{LEMMA_Compon_Bijec} that $|\rmC(G, F)|\le\kappa$.

    Now assume $|\rmC(G, F)|\le\kappa$ for every finite set of timid vertices $F$ in $G$. 
    We will show that $|\rmC(G, F)|\le\kappa$ for every finite $F\subset \mathrm{V}(H_G)$, which will conclude the proof in view of  \myref{theorem:lindelofDegree} and \myref{THM_HGHomeo}. 
    So let such $F\subset \mathrm{V}(H_G)$ be given. 
    Consider 
    \[F' = F\cap \mathrm{t}(G).\]

    Then $|\rmC(G, F')|\le\kappa$ and, by  \myref{LEMMA_Compon_Bijec}, $|\rmC(H_G, F')|\le\kappa$ as well. 
    Furthermore, it follows from \myref{PROP_Dominant_FinComp} that the additional finite $F\setminus F'$ vertices can only separate finitely many components of $\rmC(H_G, F')$ into more finite components in $\rmC(H_G, F)$.
    Hence, $|\rmC(H_G, F)|\le \kappa$ and the proof is complete.
\end{proof}

We end this section with a remark: note that, in both \myref{THM_compactnesscharacterization} and \myref{COR_LindelofEdgeEnds}, the conditions are given in terms of finite sets of \emph{timid} vertices, while \myref{COR_RothEdgeEnds}(d) is given in terms of finite sets of \emph{edges}. 

Indeed, it is clear that a simple star of $\kappa$-many rays would satisfy the conditions of \myref{THM_compactnesscharacterization} if we switch the finite sets of timid vertices for finite sets of edges, while the corresponding edge-end space would have its Lindel\"of degree equal to $\kappa$ (thus, the timid vertices play a role in \myref{THM_compactnesscharacterization} and \myref{COR_LindelofEdgeEnds} that edges cannot).

On the other hand, the following example shows that condition \myref{COR_RothEdgeEnds}(d) cannot be given in terms of finite sets of timid vertices.

\begin{example}\label{EX_NotRothEdgeEnds}
    For each $s\in 2^{<\omega}$, let $K_s$ denote the complete graph with $\aleph_0$-many vertices. 
    Let $G$ be the graph obtained from the disjoint union of all $K_s$'s with $2^{<\omega}$, with the addition of an edge between each vertex of $K_s$  with its corresponding $s\in 2^{<\omega}$.

    In this case, it is clear that every vertex of $G$ edge-dominates some ray, so that $\mathrm{t}(G) = \emptyset$.
    Thus, $G$ trivially satisfies
    \begin{itemize}
        \item[(d')] For every strictly $\subseteq$-increasing sequence $(F_n:n\in\omega)$ of finite subsets of $\mathrm{t}(G)$ there is a sequence $(C_n:n\in\omega)$, where each $C_n$ is a connected component of $G\setminus F_n$, so that every ray of $G$ has a tail in $C_n$ for some $n\in\omega$.
    \end{itemize}
    Nevertheless, as $\Omega_E(G)$ contains an obvious copy of the Cantor space, it follows from \myref{COR_RothEdgeEnds} that it is not Rothberger.
\end{example}
\begin{figure}[ht!]
    \centering
    \begin{tikzpicture}
        \def\radius{0.15}
        \def\dy{0.4}
        \def\dx{9}
        \def\x{0}
        \def\y{0}

        \node(empty) at (0,0) {$\seq{\,}$};
        \node(0) at (-4,4) [circle]{(0)};
        \node(1) at (4,4) [circle, inner sep=3pt]{(1)};
        \node(00) at (-6,8) [circle, inner sep=3pt]{(00)};
        \node(01) at (-2,8) [circle, inner sep=3pt]{(01)};
        \node(10) at (2,8) [circle, inner sep=3pt]{(10)};
        \node(11) at (6,8) [circle, inner sep=3pt]{(11)};
        \node(bindots) at (0,10) {\Large $\vdots$};

        \draw (empty) -- (0);
        \draw (empty) -- (1);
        \draw (0) -- (00);
        \draw (0) -- (01);
        \draw (1) -- (10);
        \draw (1) -- (11);

                \draw[-{Latex[length=1.5mm]}] (-2,0) -- (-2,3.25*\dy);
                \foreach \i in {0,...,2}{
                    \node (K\i) at (-2,\dy*\i) {\tiny $\bullet$};
                }
                \node (KdotsL) at (-2-0.25,\dy*3) {\tiny $\vdots$};
                \draw (-2,\dy*0) edge[bend left=60] (-2,\dy*2);
                \draw[dashed] (-2-0.75,3.5*\dy)  rectangle (-2+0.5,\dy*0-0.1);
                \node (K) at (-2-0.4,\dy*0+0.1) {\tiny $K_{\seq{\,}}$};

                \draw[-{Latex[length=1.5mm]}] (-2-4,0+4) -- (-2-4,3.25*\dy+4);
                \foreach \i in {0,...,2}{
                    \node (0K\i) at (-2-4,\dy*\i+4) {\tiny $\bullet$};
                }
                \node (0KdotsL) at (-2-0.25-4,\dy*3+4) {\tiny $\vdots$};
                \draw (-2-4,\dy*0+4) edge[bend left=60] (-2-4,\dy*2+4);
                \draw[dashed] (-2-0.75-4,3.5*\dy+4)  rectangle (-2+0.5-4,\dy*0-0.1+4);
                \node (0K) at (-2-0.4-4,\dy*0+0.1+4) {\tiny $K_{\seq{0}}$};

                \draw[-{Latex[length=1.5mm]}] (-2+4,0+4) -- (-2+4,3.25*\dy+4);
                \foreach \i in {0,...,2}{
                    \node (1K\i) at (-2+4,\dy*\i+4) {\tiny $\bullet$};
                }
                \node (1KdotsL) at (-2-0.25+4,\dy*3+4) {\tiny $\vdots$};
                \draw (-2+4,\dy*0+4) edge[bend left=60] (-2+4,\dy*2+4);
                \draw[dashed] (-2-0.75+4,3.5*\dy+4)  rectangle (-2+0.5+4,\dy*0-0.1+4);
                \node (1K) at (-2-0.4+4,\dy*0+0.1+4) {\tiny $K_{\seq{1}}$};

                \draw[-{Latex[length=1.5mm]}] (-2-6,0+8) -- (-2-6,3.25*\dy+8);
                \foreach \i in {0,...,2}{
                    \node (00K\i) at (-2-6,\dy*\i+8) {\tiny $\bullet$};
                }
                \node (00KdotsL) at (-2-0.25-6,\dy*3+8) {\tiny $\vdots$};
                \draw (-2-6,\dy*0+8) edge[bend left=60] (-2-6,\dy*2+8);
                \draw[dashed] (-2-0.80-6,3.5*\dy+8)  rectangle (-2+0.5-6,\dy*0-0.1+8);
                \node (00K) at (-2-0.45-6,\dy*0+0.1+8) {\tiny $K_{\seq{00}}$};

                \draw[-{Latex[length=1.5mm]}] (-2-2,0+8) -- (-2-2,3.25*\dy+8);
                \foreach \i in {0,...,2}{
                    \node (01K\i) at (-2-2,\dy*\i+8) {\tiny $\bullet$};
                }
                \node (01KdotsL) at (-2-0.25-2,\dy*3+8) {\tiny $\vdots$};
                \draw (-2-2,\dy*0+8) edge[bend left=60] (-2-2,\dy*2+8);
                \draw[dashed] (-2-0.80-2,3.5*\dy+8)  rectangle (-2+0.5-2,\dy*0-0.1+8);
                \node (01K) at (-2-0.45-2,\dy*0+0.1+8) {\tiny $K_{\seq{01}}$};

                \draw[-{Latex[length=1.5mm]}] (-2+2,0+8) -- (-2+2,3.25*\dy+8);
                \foreach \i in {0,...,2}{
                    \node (10K\i) at (-2+2,\dy*\i+8) {\tiny $\bullet$};
                }
                \node (10KdotsL) at (-2-0.25+2,\dy*3+8) {\tiny $\vdots$};
                \draw (-2+2,\dy*0+8) edge[bend left=60] (-2+2,\dy*2+8);
                \draw[dashed] (-2-0.80+2,3.5*\dy+8)  rectangle (-2+0.5+2,\dy*0-0.1+8);
                \node (10K) at (-2-0.45+2,\dy*0+0.1+8) {\tiny $K_{\seq{10}}$};

                \draw[-{Latex[length=1.5mm]}] (-2+6,0+8) -- (-2+6,3.25*\dy+8);
                \foreach \i in {0,...,2}{
                    \node (11K\i) at (-2+6,\dy*\i+8) {\tiny $\bullet$};
                }
                \node (11KdotsL) at (-2-0.25+6,\dy*3+8) {\tiny $\vdots$};
                \draw (-2+6,\dy*0+8) edge[bend left=60] (-2+6,\dy*2+8);
                \draw[dashed] (-2-0.80+6,3.5*\dy+8)  rectangle (-2+0.5+6,\dy*0-0.1+8);
                \node (11K) at (-2-0.45+6,\dy*0+0.1+8) {\tiny $K_{\seq{11}}$};

                    \foreach \i in {0,...,2}{
                        \draw[line width=2pt, white] (empty) -- (K\i);
                        \draw (empty) edge (-2,\dy*\i);
                    }
                    \node(KRdots) at (-1.25,2.25*\dy) {\tiny $\vdots$};
                    \foreach \i in {0,...,2}{
                        \draw[line width=2pt, white] (0) -- (0K\i);
                        \draw (0) edge (-2-4,\dy*\i+4);
                    }
                    \node(0KRdots) at (-1.25-4,2.25*\dy+4) {\tiny $\vdots$};
                    \foreach \i in {0,...,2}{
                        \draw[line width=2pt, white] (1) -- (1K\i);
                        \draw (1) edge (-2+4,\dy*\i+4);
                    }
                    \node(1KRdots) at (-1.25+4,2.25*\dy+4) {\tiny $\vdots$};
                    \foreach \i in {0,...,2}{
                        \draw[line width=2pt, white] (00) -- (00K\i);
                        \draw (00) edge (-2-6,\dy*\i+8);
                    }
                    \node(00KRdots) at (-1.25-6,2.25*\dy+8) {\tiny $\vdots$};
                    \foreach \i in {0,...,2}{
                        \draw[line width=2pt, white] (01) -- (01K\i);
                        \draw (01) edge (-2-2,\dy*\i+8);
                    }
                    \node(01KRdots) at (-1.25-2,2.25*\dy+8) {\tiny $\vdots$};
                    \foreach \i in {0,...,2}{
                        \draw[line width=2pt, white] (10) -- (10K\i);
                        \draw (10) edge (-2+2,\dy*\i+8);
                    }
                    \node(10KRdots) at (-1.25+2,2.25*\dy+8) {\tiny $\vdots$};
                    \foreach \i in {0,...,2}{
                        \draw[line width=2pt, white] (11) -- (11K\i);
                        \draw (11) edge (-2+6,\dy*\i+8);
                    }
                    \node(11KRdots) at (-1.25+6,2.25*\dy+8) {\tiny $\vdots$};

            \draw[thick] (-9,10.75)  rectangle (6.5,-0.5);
            \node (G) at (6,10.25) {\Large $G$};
    \end{tikzpicture}
    \caption{Figure of the graph in \myref{EX_NotRothEdgeEnds}, in which no vertex is timid but whose edge-end space is not Rothberger.}
    \label{FIG_NotQuotient}
\end{figure}
\section{Final remarks}

In this work, we have investigated several covering properties of ray spaces and end spaces of graphs, elucidating their connections to the combinatorial structure of trees and graphs.
Our results show that ray spaces form a new and natural class of paracompact $D$-spaces, and we have characterized the Lindelöf degree and the extent of end spaces and edge-end spaces in terms of the underlying graph's combinatorics.
We have also provided a combinatorial characterization of the Rothberger property for ray spaces, relating it to the absence of Cantor subspaces and the scatteredness of the space.
Furthermore, we established that these spaces are Menger if and only if they are $\sigma$-compact, and gave precise criteria for when this occurs.

These findings contribute to a deeper understanding of the interplay between topological properties and graph-theoretic structure, and suggest several directions for future research.
However, we feel like there might exist a simpler characterization for $\sigma$-compactness of end spaces.
Thus, we pose the following question.

\begin{question}
    Is there a simpler combinatorial characterization of when the end space of a graph is $\sigma$-compact?
\end{question}

The authors thank Max Pitz for pointing out Corollary~5.5 of \cite{koloschin2023end}, which made the proof of Proposition~\ref{proposition:sigmaCompactEnds} trivial and directly answered a question posed in the first draft of this manuscript.
\printbibliography
\end{document}